\documentclass[12pt]{amsart}

\usepackage{latexsym, amssymb}

\newcommand{\be}{\begin{equation}}
\newcommand{\ee}{\end{equation}}
\newcommand{\beq}{\begin{eqnarray}}
\newcommand{\eeq}{\end{eqnarray}}
\newcommand{\bee}{\begin{equation*}}
\newcommand{\eee}{\end{equation*}}
\newtheorem{thm}{Theorem}[section]

\newtheorem{lma}{Lemma}[section]
\newtheorem{prop}{Proposition}[section]
\newtheorem{cor}{Corollary}[section]
\theoremstyle{definition}

\theoremstyle{remark}
\newtheorem{rem}{Remark}[section]
\numberwithin{equation}{section}

\def\ep{\epsilon}
\def\E{\mathcal{E}}
\def\p{\partial}
\def\S{\Sigma}

\def\R{\mathbb{R}}

\def\Sp{\Sigma^\prime}

\def\p{\partial}
\def\lf{\left}
\def\ri{\right}
\def\e{\epsilon}

\def\R{\Bbb R}
\def\wt{\widetilde}
\def\la{\langle}
\def\ra{\rangle}

\def\hs{\hat\sigma}

\def\Ric{\text{\rm Ric}}
\def\div{\text{\rm div}}

\def\vh{\vspace{.3cm}}

\def\Pi{\displaystyle{\mathbb{II}}}

\def\E{\mathcal{E}}

\def\Sp{\mathcal{S}}
\def\X{\mathcal{X}}
\def\M{\mathcal{M}}
\def\MK{\mathbb{R}^{3,1}}
\def\mS{\mathcal{S}}

\def\mH{\mathcal{H}}
\def\hS{\hat \Sigma}
\def\mly{\mathfrak{m}_{_{LY}}}
\def\mby{\mathfrak{m}_{_{BY}}}
\def\mwy{\mathfrak{m}_{_{WY}}}
\def\Ewy{E_{_{WY}}}
\def\stop{\hfill$\Box$}
\def\factor{\frac1{8\pi}}

\def\Sp{\mathcal{S}}
\def\X{\mathcal{X}}
\def\M{\mathcal{M}}

\def\s{\sigma}
\def\bA{\mathbf{A}}
\def\bB{\mathbf{B}}
\def\bC{\mathbf{C}}
\def\bP{\mathbf{P}}

\begin{document}

\title[]{Critical points of  Wang-Yau quasi-local energy}

\author{Pengzi Miao$^1$}
\address[Pengzi Miao]{School of Mathematical Sciences, Monash University, Victoria 3800, Australia;
Department of Mathematics, University of Miami, Coral Gables, FL 33146, USA}
\email{Pengzi.Miao@sci.monash.edu.au, 
pengzim@math.miami.edu}

\author{Luen-Fai Tam$^2$}
\address[Luen-Fai Tam]{The Institute of Mathematical Sciences and Department of
 Mathematics, The Chinese University of Hong Kong,
Shatin, Hong Kong, China.}
\email{lftam@math.cuhk.edu.hk}

\author{Naqing Xie$^3$}
\address[Naqing Xie]{School of Mathematical Sciences, Fudan
University, Shanghai 200433, China}
\email{nqxie@fudan.edu.cn}

\thanks{$^1$ Research partially supported by Australian Research Council Discovery Grant  \#DP0987650}
\thanks{$^2$Research partially supported by Hong Kong RGC General Research Fund  \#GRF 2160357}
\thanks {$^3$Research partially supported by the National Science
Foundation of China \#10801036
and the Innovation Program of Shanghai Municipal Education Commission
\#11zz01}

\renewcommand{\subjclassname}{
  \textup{2010} Mathematics Subject Classification}
\date{}

\begin{abstract}
In this paper, we prove the following theorem regarding
the Wang-Yau quasi-local energy of a spacelike two-surface in a spacetime:
Let $ \Sigma $ be a boundary component of
some compact, time-symmetric, spacelike hypersurface $ \Omega $
in a time-oriented spacetime $ N$ satisfying the dominant energy condition.
Suppose the induced metric  on $ \Sigma $
 has positive Gaussian  curvature and
all boundary components of $ \Omega $ have positive mean curvature.
Suppose
$
 H \le H_0
$
where $ H$ is   the mean curvature of $ \Sigma $ in $ \Omega$ and $ H_0$ is
the mean curvature of $ \Sigma$ when isometrically embedded in $ \R^3$.
If $ \Omega $ is not isometric to a domain in $ \R^3$, then
\begin{enumerate}

 \item the Brown-York mass  of $ \Sigma $ in $ \Omega$
 is  a strict local minimum of the Wang-Yau quasi-local energy of $ \Sigma$.

  \item on a small perturbation $ \tilde{\Sigma} $ of $ \Sigma $ in $ N$, there exists
 a critical point of the Wang-Yau quasi-local energy of $ \tilde{\Sigma}$.

 \end{enumerate}
\end{abstract}

\date{February, 2011}

\maketitle

\markboth{Pengzi Miao, Luen-Fai Tam, and Naqing Xie}
{Critical points of  Wang-Yau quasilocal energy}

\section{Introduction and statement of the result} \label{Introduction}

Let $ N $ be a space-time, i.e. a Lorentzian manifold of dimension four.
Suppose $ N $ is time orientable.
Denote the Lorentzian metric on $ N$ by $ \la \cdot , \cdot \ra$ and its covariant
derivative by $ \nabla^N$.
Let $ \S \subset N $ be an embedded, spacelike two-surface that is topologically a two-sphere.
Suppose the mean curvature vector $ H $ of $ \S $ in $ N $ is spacelike.
Let $ \sigma $ be the induced metric on $ \S$ and let $ K$ be the Gaussian curvature of $ (\S, \sigma)$.

 Given a function $ \tau $ on $ \S$ such that $ \hs = \sigma + d \tau \otimes d \tau $ is a metric  of positive Gaussian curvature on $ \S$, by  \cite[Theorem 3.1]{WangYau08}
 there exists an isometric embedding $X: (\Sigma,\sigma) \hookrightarrow \MK$ such that
 $ \tau $ is the time function of $ X$, i.e.
 $ X = (\hat{X}, \tau) , $
 where $ \hat{X} = (\hat{X}_1, \hat{X}_2, \hat{X}_3 )$ is an isometric
 embedding of $ (\Sigma, \hat{\sigma})$ in $ \R^3 = \{(x,0)\in \R^{3,1}\}$.
 The Wang-Yau quasi-local energy  \cite{WangYau-PRL, WangYau08}, associated to such a time function $ \tau$,  is given by
\begin{equation*} \label{wy-energy}
\begin{split}
\ &  \Ewy(\Sigma,\tau)   \\
=& \ \frac1{8\pi} \left\{
\int_{\hS}\hat H dv_{\hS}-\int_\Sigma\lf[ \sqrt{1+|\nabla\tau|^2}\cosh\theta |H|-\la\nabla\tau,\nabla \theta \ra-\la V,\nabla\tau\ra\ri]dv_\Sigma \right\} ,
\end{split}
\end{equation*}
where
\begin{itemize}
    \item $ \hS = \hat X ( \S) $, $ \hat H > 0 $ is the mean curvature of $ \hS $ in $ \R^3$,  $ d v_{\hS} $ and $ d v_\S$ are the volume forms of the metrics $ \hs $ and $ \sigma$.
    \item  $\sinh\theta=\frac{-\Delta \tau}{|H|\sqrt{1+|\nabla\tau|^2}}$, $ \nabla $ and $ \Delta$ are the gradient and the Laplacian operators of the metric $ \sigma$,  $ |  H | = \sqrt{ \la H, H \ra }$.
   \item $V$ is the tangent vector on $ \Sigma $ that is dual to the one form $\alpha^N_{e_3^{H}} (\cdot ) $ defined by $\alpha^N_{e_3^{H}}(X)=\la \nabla^N_X e_3^{H}, e_4^{H}\ra$ for any $ X $ tangent to $ \S$.
       Here $e_3^H=- \frac{H}{ |H| }$ and $e_4^H$ is the future timelike unit normal to $     \S$ that is orthogonal to $e_3^H$.
        \end{itemize}
The Wang-Yau quasi-local mass of $ \S$  \cite{WangYau-PRL, WangYau08}, which we denote by $\mwy(\Sigma)$,
is then defined to be
$$
\mwy(\Sigma) = \inf_{ \tau} \Ewy (\Sigma, \tau)
$$
where the infimum is taken over all functions $ \tau $ that are {\it admissible}
(see \cite[Definition 5.1]{WangYau08} for the definition of admissibility).

In  \cite{WangYau08}, Wang and Yau show that  a function $ \tau $ is a critical point of
$ \Ewy(\S, \cdot) $ if and only if  $\tau$ satisfies
 \be
\label{1stvariation-e1}
-  \lf[ \hat{H}\hat{\sigma}^{ab} - \hat{\sigma}^{ac} \hat{\sigma}^{bd} \hat{h}_{cd} \ri]
\frac{ \nabla_b \nabla_a \tau }{ \sqrt{ 1 + | \nabla \tau |^2 } }
+ \div_{\Sigma} \lf[ \frac{\cosh \theta |  H | }{\sqrt{ 1 + | \nabla \tau |^2} }\nabla \tau - \nabla \theta - V \ri]=0 ,
\ee
where $ \hs$, $ \hat H$, $ \theta$ and $ V $ are defined as above,
$\{a, b, c, d\}$ denote indices of local coordinates on $ \S$, $ \hat{h}_{ab} $ is the second fundamental form of
$ \hat \S $ in $ \R^3$ and $ \div_\Sigma (\cdot) $ denotes the divergence operator on $(\S, \sigma)$.

When the Gaussian curvature $ K $ of $(\Sigma, \sigma)$ is positive, the function
$ \tau_0 = 0 $ is admissible \cite[Remark 1.1]{WangYau08} and
$
\Ewy (\Sigma, \tau_0) = \mly (\Sigma),
$
where $ \mly (\Sigma) $ is the Liu-Yau quasi-local mass of $ \Sigma$ \cite{LY1, LY2}.
In this case, $ \tau_0 $ is a critical point of $ \Ewy(\Sigma, \cdot)$ if and only if
$ \div_\Sigma V = 0 $.

Now suppose  $ \Sigma $ is one of the boundary components of  a compact, time-symmetric, space-like hypersurface $ \Omega $ in $ N$,  then $ V = 0 $ and
 $ \Ewy (\Sigma, \tau_0 ) = \mby (\Sigma, \Omega) $, where  $ \mby(\Sigma, \Omega) $ is the Brown-York mass of $ \Sigma $ in $\Omega$ \cite{BY1, BY2}.
Considering the variational nature of $ \mwy(\Sigma)$,  one naturally wants to ask the following:

\vh

\noindent {\bf Question 1.} Suppose $ \Sigma $ is  a boundary component of  a compact, time-symmetric, space-like hypersurface $ \Omega $ in $ N$,  is the Brown-York mass $\mby(\Sigma,\Omega)$ a  local minimum value of the Wang-Yau quasi-local energy
 $ \Ewy (\Sigma, \cdot) $?

\vh

\noindent {\bf Question 2.}  Suppose $ \Sigma $ is  a boundary component of  a compact, time-symmetric, space-like hypersurface $ \Omega $ in $ N$,
is the set of solutions to \eqref{1stvariation-e1} open near the pair $(\Sigma,\tau_0)$? That is, suppose $ \tilde{\Sigma} \subset N $ is another closed, embedded,
spacelike two-surface which is a small perturbation of $ \Sigma$, does there
exist a solution $ \tau $ to \eqref{1stvariation-e1} with $ \Sigma$ replaced by $ \tilde{\Sigma}$?

\vh

Our main result in this paper is the following theorem:

\begin{thm}\label{Mainresult}
Let $ \Sigma $ be a boundary component of
some compact, time-symmetric, spacelike hypersurface $ \Omega $
in a time-oriented spacetime $ N$ satisfying the dominant energy condition.
Suppose the induced metric $ \sigma $ on $ \Sigma $
 has positive Gaussian  curvature and
all boundary components of $ \Omega $ have positive mean curvature.
Suppose
\be \label{Hcondition}
 H \le H_0
 \ee
where $ H$ is   the mean curvature of $ \Sigma $ in $ \Omega$ and $ H_0$ is
the mean curvature of $ \Sigma$ when isometrically embedded in $ \R^3$.
If $ \Omega $ is not isometric to a domain in $ \R^3$, then
\begin{enumerate}
 \item $\mby(\Sigma,\Omega)$ is a strict local minimum of $\Ewy(\Sigma,\cdot)$.

 \item  for $\tilde \Sigma\subset N$ near $\Sigma$, there is a solution $\tau$ to \eqref{1stvariation-e1} for $\tilde \Sigma$.

\end{enumerate}
\end{thm}

\vspace{.1cm}

We note that there are many types of surfaces $ \Sigma $ that satisfy the condition \eqref{Hcondition} of Theorem \ref{Mainresult}.  Here we list a few of them:

\begin{enumerate}

 \item[(i)]  $ \Sigma = S_r $,  where $  S_r = \{ | x | = r \}$ is
a large coordinate sphere in
 a time-symmetric,  {\em asymptotically Schwarzschild} ({\bf AS}), spacelike slice $ M \subset N $.
 Here a three-Riemannian manifold $ M $  is  called {\bf AS} (with mass $m$)
 if  there is a compact set $ K \subset M$ such that
  $ M \setminus K $ is diffeomorphic to $ \R^3 \setminus \{ | x | \le R \} $ for some $ R $
   and the metric $g$ on $ M$ with respect to the standard coordinates on $ \R^3 $ takes the form
$$  g_{ij} = \lf( 1 + \frac{ m}{  2r} \ri)^4 \delta_{ij} + b_{ij} $$
where $ |\p^k b_{ij}| = O \lf( r^{-2 - k } \ri), $  $0\le k \le 3$, $r=|x|$ and $ m $ is a constant.
Direct calculation  (see (5.1) in \cite{Huang08-centermass} for example) gives
$$
H = \frac{2}{r} - \frac{4m}{r^2} + O ( r^{-3} ) .
$$
On the other hand, it was proved in \cite{ShiTam02} (the equation on the bottom of page 122) that
$$
H_0 = \frac{2}{r} - \frac{2m}{r^2} + O ( r^{-3} ) .
$$
Therefore,  $ H < H_0$ for large $r$ if $ M $ has positive mass $ m$.

\vspace{.1cm}

  \item [(ii)]  $ \Sigma $ bounds a compact, time-symmetric
  spacelike slice $ \Omega$ and  $ \Sigma $ has constant positive Gaussian curvature and
 constant positive mean curvature $ H$.  In this case, by the results in   \cite{Miao02, ShiTam02}
 one knows $ H \leq  H_0 $ and $ H = H_0 $ if and only if $ \Omega $ is isometric to a Euclidean
  round ball.

\vspace{.1cm}

  \item[(iii)]    $ \Sigma $ bounds a compact, time-symmetric
  spacelike slice $ \Omega$ and $ \Sigma $ has positive Gaussian curvature and
  positive mean curvature.
 Suppose   there exists a conformal
  diffeomorphism $ f: \Omega \rightarrow \Omega_0$ between $\Omega $
  and a domain $ \Omega_0 $ in $ \R^3 $ such that
  $ f^* ( g_0) $ and $ g $ induce the same boundary metric on $ \Sigma $ and
  $ \Sigma $ has positive mean curvature in $(\Omega, f^*( g_0) )$. Here
  $ g $ is the metric on $ \Omega$ and   $ g_0 $ is the
  Euclidean metric on $ \Omega_0$.
  In this case, if one writes $ g = u^4 f^* ( g_0 ) $,
   it follows from the maximum principle (applied to $ u $)
   that $ H \le  H_0$ on $ \Sigma$ and $ H = H_0 $ precisely when
   $ \Omega $ is isometric to  $ \Omega_0$.

\vspace{.1cm}

\item[(iv)]  When viewed purely as a result on the Riemannian $3$-manifold $ \Omega$,
Theorem \ref{Mainresult} applies to those $ \Omega$ that are graphs over convex Euclidean domains.
Precisely, let $ \Sigma $ be a strictly convex closed surface in $ \R^3$ and let $ \Omega_0 \subset \R^3 $
be  its interior.  Let $ f: \Omega_0 \rightarrow \R$ be a smooth function such that $ f |_{\Sigma} = 0 $.
Let $ \Omega $ be the graph of $f $ in $ \R^4$ with the induced metric and let $ H $ be the mean curvature
of $ \Sigma $ in $ \Omega$. Directly calculation shows $  H = \frac{1}{\sqrt{ 1 +  | \nabla f |^2 } } H_0 \le H_0 $.
   The motivation to consider these $\Omega$
   (with $ f $ chosen such that $ \Omega$ has nonnegative scalar curvature)
   comes from a recent work of Lam \cite{Lam2010}
   on the graphs cases of the Riemannian positive mass theorem and Penrose inequality.

\end{enumerate}

\vspace{.2cm}

We should mention that related to (i) above,  Chen-Wang-Yau   \cite[Section 4]{ChenWangYau}
under the assumption of analyticity
 show that in asymptotically flat space-times, \eqref{1stvariation-e1} has a
formal power series solution, which is  locally energy minimizing at all orders, for
certain surfaces in  an asymptotically flat hypersurface.

This paper is organized as follows:
In Section \ref{localmin}, we compute the second variation  of $ \Ewy(\Sigma, \tau )$ at $ \tau_0 = 0 $ and derive a sufficient condition for $ \mby(\Sigma, \Omega)$  to locally minimize $ \Ewy(\Sigma, \tau)$. In Section \ref{moreonHcondition}, we prove that the sufficient condition provided in Section \ref{localmin}  holds for those surfaces $ \Sigma $ satisfying the assumptions in Theorem \ref{Mainresult}. Hence, part (1) of Theorem \ref{Mainresult} follows from Section \ref{localmin} and \ref{moreonHcondition}. We note that, besides playing a key role in the proof of Theorem \ref{Mainresult}, Theorem \ref{Mainresult-1} in Section \ref{moreonHcondition}  concerns analytical features of the boundary of  compact Riemannian manifolds with nonnegative scalar curvature, thus is of independent interest.
In Sections \ref{secondfform}  and \ref{smallsolution},
we focus on  part (2) of Theorem \ref{Mainresult}.
The main idea there is to apply the Implicit Function Theorem ({IFT}). But to apply the {IFT}, we are confronted with the problem to show that the map $ F $,
sending a metric $ \sigma $ of positive Gaussian curvature on the two-sphere $ S^2$ to the second fundamental form $ \Pi $ of the isometric embedding of $(S^2, \sigma)$ in $ \R^3$, is a $ C^1 $ map between appropriate functional spaces. If $ \sigma $ is
a $ C^{k, \alpha}$ ($k \ge 2$) metric, by \cite{Nirenberg} one knows $ \Pi $ is a $ C^{k-2, \alpha}$ symmetric tensor. We do not know whether $ F $ is
$C^1$  from the $C^{k, \alpha}$ space to the $ C^{k-2, \alpha} $ space. However,
in Section \ref{secondfform},  we prove that $ F $ is $C^1$  between $ C^{k, \alpha} $ and $ C^{k-3, \alpha}$ spaces for $ k \ge 4$.
This turns out to be sufficient to apply the {IFT} to obtain  solutions
to \eqref{1stvariation-e1} because the metric $ \hat \sigma $ in \eqref{1stvariation-e1}  involves $ d \tau $ and \eqref{1stvariation-e1} is a 4-th order differential equation of the function $ \tau$. In Section \ref{smallsolution}, we apply the result in
Sections \ref{moreonHcondition}, \ref{secondfform} and the {IFT} to prove  the existence of critical points of $\Ewy(\tilde{\Sigma}, \cdot) $ for surfaces $ \tilde{\Sigma} $ nearby.

We want to thank Michael Eichmair for helpful discussions leading to Proposition \ref{Reilly-t}.

\section{Comparing $ \mby(\Sigma, \Omega) $ and $ \Ewy(\Sigma, \cdot) $}\label{localmin}


We start this section by  computing the second variation of $ \Ewy(\Sigma, \cdot)$ at $ \tau_0=0 $,
assuming $ \tau_0 $ is a critical point for $ \Ewy(\Sigma, \cdot)$.

\begin{prop}\label{2ndvar-l1}
Let $ N $ be a time-oriented spacetime.
Let $ \Sigma \subset N $ be an embedded, spacelike two-surface that is topologically a two-sphere.
Suppose the mean curvature vector $ H $ of $ \S $ in $ N $ is spacelike.
If $ \tau_0 = 0 $ is a critical point for  $ \Ewy(\Sigma, \cdot)$, then  the second variation of $\Ewy(\Sigma, \cdot)$ at $\tau_0=0$ is given by
  \be \label{2ndvarformula}
\begin{split}
  & \delta^2  E_{_{WY}} ( \Sigma, \tau ) |_{ \tau = 0 }(\delta\tau)\\
   & =\factor \int_\Sigma \lf[ \frac{ ( \Delta \delta\tau)^2 }{ | H |  } + ( H_0 - | H | ) | \nabla (\delta\tau) |^2 - \Pi_0 ( \nabla \delta\tau ,\nabla \delta\tau)\ri]dv_{\Sigma }
\end{split}
\ee
where $H_0$ and $\Pi_0$ are the mean curvature and the second fundamental form of $(\Sigma, \sigma) $ when isometrically embedded in $\R^3$, and $\sigma$ is the induced metric on $\Sigma$ from $N$.
  \end{prop}

\begin{proof} The first variation of $ \Ewy(\Sigma, \cdot) $ was obtained by
Wang and Yau in \cite[Proposition 6.2]{WangYau08} and is given by
\be\label{1stvar-wy}
 \begin{split}
  \delta \Ewy( \Sigma, \tau) ( \delta \tau) & =
  \frac{1}{8 \pi } \int_\Sigma
 \bigg\{ - \lf[  \hat H   { \hat \sigma}^{ab}-  { \hat \sigma}^{ac} { \hat \sigma}^{bd} ( \hat h_{cd}) \ri]
\frac{ \nabla_b \nabla_a  \tau }{ \sqrt{ 1 + | \nabla \tau |^2 } }\\
& \
+ \div_\Sigma \lf[ \frac{ \nabla  \tau}{\sqrt{ 1 +  | \nabla \tau |^2} } \cosh \theta  | H |   - \nabla \theta
- V  \ri] \bigg\} \cdot \delta \tau  \ d v_{\Sigma} .
\end{split}
\ee
Let $ \mH (\tau )$ denote the functional
\be
\begin{split}
 - \lf[  \hat H   { \hat \sigma}^{ab}-  { \hat \sigma}^{ac} { \hat \sigma}^{bd} ( \hat h_{cd}) \ri]
\frac{ \nabla_b \nabla_a  \tau }{ \sqrt{ 1 + | \nabla \tau |^2 } }
+ \div_\Sigma \lf[ \frac{ \nabla  \tau}{\sqrt{ 1 +  | \nabla \tau |^2} } \cosh \theta  | H |   - \nabla \theta
- V  \ri] .
\end{split}
\ee
Direct computation shows that the first variation of
$ \mH (\cdot) $ at $ \tau = 0 $ is
\be \label{1stvar-htau}
\delta \mH ( \tau ) |_{\tau = 0 } ( \delta \tau ) =
- \la H_0 \sigma - \Pi_0, \nabla^2 \delta \tau \ra
+ \div_\Sigma \lf(  | H |  \nabla  \delta \tau \ri)
 +  \Delta \lf( \frac{ \Delta \delta \tau }{ | H | } \ri),
\ee
where $ \nabla^2 $ denotes the Hessian  operator on $(\Sigma, \sigma)$.
\eqref{2ndvarformula} now follows from \eqref{1stvar-wy}, \eqref{1stvar-htau}
and the fact that $ H_0 \sigma - \Pi_0 $ is divergence free on $ (\Sigma, \sigma)$.
\end{proof}

Assuming the quadratic functional of $ \delta \tau $ in \eqref{2ndvarformula}
has certain positivity property, we show that
 $ \tau_0 = 0 $ is a strict local minimum point
for $ \Ewy (\Sigma, \cdot)$.

\begin{thm}\label{minimum-t1}
Let $ \Sigma $ be a boundary component of
some compact, time-symmetric, spacelike hypersurface $ \Omega $
in a time-oriented spacetime $ N$ satisfying the dominant energy condition.
 Suppose the induced metric $ \sigma $ on $ \Sigma $
 has positive Gaussian  curvature  and the mean curvature $ H$ of $ \Sigma $
in $ \Omega $  is positive.
Suppose in addition that there exists a constant $\beta>0$ such that
\be\label{assumption-e2}
 \int_\Sigma \lf[ \frac{ ( \Delta \eta )^2 }{ H  } + ( H_0 - H ) | \nabla \eta  |^2 - \Pi_0 ( \nabla \eta ,\nabla \eta)\ri]dv_{\Sigma}\ge \beta\int_\Sigma (\Delta\eta)^2dv_{\Sigma}
\ee
for all $\eta\in W^{2,2}(\Sigma)$, where $H_0$ and $\Pi_0$ are the mean curvature and the second fundamental form of $(\Sigma,\sigma)$ when isometrically embedded in $\R^3$.
Then, for  any constant $0<\alpha<1$, there exists a constant $\e > 0 $
depending only on $ \sigma$, $ H $ and $ \beta$, such that
 \be\label{minimum-e1}
 \Ewy(\Sigma,\tau) -  \mby(\Sigma,\Omega) \ge \frac{\beta}{4}  \int_\Sigma ( \Delta \tau)^2 d v_\Sigma
\ee
 for any smooth function $\tau$ with $||\tau||_{C^{3,\alpha}}<\e$.
 \end{thm}

\begin{proof}
Let $ X(\sigma)$ be a fixed isometric embedding of $(\Sigma, \sigma)$
in $ \R^3$.  By \cite[p.353]{Nirenberg},  there exist positive constants $C_1 $
and $\e_1$, depending only on $\sigma $, such that if $\tilde{\sigma}$ is another $C^{2,\alpha}$ metric on $\Sigma$ with $||\tilde{\sigma}-\sigma ||_{C^{2,\alpha} } <\e_1 $, then
$ \tilde{\sigma} $ has positive Gaussian curvature and there exists an isometric embedding $X(\tilde{\sigma})$ of $(\Sigma, \tilde{\sigma})$ in $\R^3$ such that
\be\label{minimum-e2}
||X(\tilde{\sigma})-X(\sigma)||_{C^{2,\alpha}}\le C_1  ||\tilde{\sigma}-\sigma||_{C^{2,\alpha}}  .
\ee

Now,   let $\tau$ be any given smooth function with $||\tau||_{C^{3,\alpha}}^2<\e_1 $.
Let $\sigma(s)=\sigma +s^2d\tau\otimes d\tau$, $ 0 \le s \le 1$. Then
$$ || \sigma(s) - \sigma ||_{C^{2,\alpha}} \le || d \tau \otimes d \tau ||_{C^{2, \alpha}}
\le || \tau ||_{C^{3, \alpha}}^2 < \e_1  .$$
Hence, $\sigma(s)$ has positive Gaussian curvature and there exists an isometric embedding $X(s)$ of $( \Sigma, \sigma(s) )$ in $\R^3$ such that
\be\label{embeddings-e1}
||X(s)-X(0)||_{C^{2,\alpha}}\le C_1  ||\tau||_{C^{3,\alpha}}^2
\ee
where $X(0)=X(\sigma)$. Let $H_0(s)$ and $\Pi_0(s)$ be the mean curvature and the second fundamental form of $X(s) (\Sigma) $. Let $ d v_{\sigma(s)} $  be the volume  form of $ \sigma(s)$.  For simplicity, denote
$\Ewy(\Sigma,\sigma(s))$ by $\Ewy(s)$. By \eqref{1stvar-wy} (and also the fact $ V = 0$), we have
\be\label{1stvar-e1}
 \begin{split}
  \frac{d}{ds}\Ewy(s) = &  \
 \frac{1}{8 \pi}  \int_\Sigma
 \bigg\{ - \lf[  {H_0}(s)   {\sigma}^{ab}(s) -  {\sigma}^{ac}(s) {\sigma}^{bd}(s) (\Pi_0(s))_{cd}) \ri]
\frac{ s\nabla_b \nabla_a  \tau }{ \sqrt{ 1 + s^2| \nabla \tau |^2 } }\\
& \
+ \div_\Sigma \lf[ \frac{ s\nabla  \tau}{\sqrt{ 1 + s^2| \nabla \tau |^2} } H\cosh \theta   - \nabla \theta \ri]
\bigg\} \tau  \ d v_{\Sigma}\\
= & \
\frac{1}{8 \pi}   \bigg\{ \int_\Sigma s\lf[   {H_0}(s)    {\sigma}^{ab}(s) -  {\sigma}^{ac}(s) {\sigma}^{bd}(s) (\Pi_0(s))_{cd}) \ri] \tau_a\tau_b  \sqrt{ 1 + s^2| \nabla \tau |^2 }    \ d v_{\Sigma}\\
& \ -\int_\Sigma \frac {s|\nabla \tau|^2}{\sqrt{ 1 + s^2| \nabla \tau |^2} }H\cosh \theta \ d v_{_\Sigma}-\int_\Sigma\theta\Delta \tau \ d v_{\Sigma}  \bigg\}
\end{split}
\ee
where we have used the facts that $ {H_0}(s)   {\sigma} (s) -  \Pi_0 (s)$ is divergence free with respect to $\sigma(s)$,
$dv_{\sigma(s)}= (1+s^2|\nabla \tau|^2)^\frac12\, dv_{\Sigma}$ and
\be
\frac{  \nabla_b \nabla_a  \eta }{  { 1 + s^2| \nabla \tau |^2 } }= \nabla_b^s  \nabla_a^s  \eta
\ee
for any function $ \eta $ on $ \Sigma$.
Here $\nabla^s$ denotes  the covariant derivative of $\sigma(s)$,
and $\theta=\theta(s)$ is the function defined by
\be
\sinh \theta =\frac{-s\Delta \tau}{H\sqrt{ 1 + s^2| \nabla \tau |^2}}.
\ee

We estimate the expression in \eqref{1stvar-e1} term by term.
First note that
\be
   \cosh\theta-1 \le  \sinh^2\theta,  \ \
    |\theta-\sinh\theta|\le |\sinh^3\theta|, \ \  \forall \ \theta \in \R.
\ee
Therefore,
\be\label{est-e1}
\begin{split}
\int_\Sigma\theta\Delta \tau \ d v_{\Sigma}=& \int_\Sigma\sinh\theta\Delta \tau \ d v_{\sigma_0}+\int_\Sigma(\theta-\sinh\theta)\Delta \tau \ d v_{\Sigma}\\
=&-\int_\Sigma \frac{s(\Delta\tau)^2}{H}  \ d v_{\Sigma}+F_1
\end{split}
\ee
where
\be
|F_1|\le C_2 s^3  ||\tau||_{C^{2,\alpha}}^2 \int_\Sigma \lf [|\nabla\tau|^2+(\Delta\tau)^2 \ri]    d v_{\Sigma}
\ee
for some constant $C_2 $ depending only on $H$. Similarly,
\be\label{est-e2}
\begin{split}
\int_\Sigma &\frac {s|\nabla \tau|^2}{\sqrt{ 1 + s^2| \nabla \tau |^2} }H\cosh \theta  \ d v_{\Sigma}\\
=& \int_\Sigma \frac {s|\nabla \tau|^2}{\sqrt{ 1 + s^2| \nabla \tau |^2} } H\ d v_{\Sigma}+\int_\Sigma \frac {s|\nabla \tau|^2}{\sqrt{ 1 + s^2| \nabla\tau |^2} }H(\cosh \theta-1)  \ d v_{\Sigma}\\
=&\int_\Sigma s|\nabla \tau|^2   H\ d v_{\Sigma}+F_2
\end{split}
\ee
where
\be
 |F_2|\le C_3 s^3 ||\tau||_{C^{2,\alpha}}^2
\int_\Sigma [ |\nabla \tau|^2 + ( \Delta \tau )^2 ] d v_{\Sigma}
\ee
for some constant $C_3$ depending only on $H$.
Next, by \eqref{embeddings-e1} we have
\be
||\Pi_0(s)-\Pi_0 ||_{C^{0,\alpha}}\le C_4 ||\tau||_{C^{3,\alpha}}^2
\ee
for some constant $C_4$ depending only on $\sigma$. This, together with the fact that $||\sigma(s)-\sigma ||_{C^{2,\alpha}}\le ||\tau||_{C^{3,\alpha}}^2$ implies
\be\label{est-e3}
\begin{split}
\int_\Sigma s   {H}_0(s)   {\sigma}^{ab}(s)  \tau_a\tau_b  \sqrt{ 1 + s^2| \nabla \tau |^2 }    \ d v_{\sigma_0}\
=& \int_\Sigma sH_0 |\nabla \tau|^2d v_{\Sigma}+F_3
\end{split}
\ee
where
\be
|F_3|\le   C_5 s ||\tau||_{C^{3,\alpha}}^2  \int_\Sigma |\nabla \tau|^2 d v_{\Sigma}
\ee
 for some constant $C_5$ depending only on $\sigma$.  Similarly,
\be\label{est-e4}
\begin{split}
\int_\Sigma s  {\sigma}^{ac}(s) {\sigma}^{bd}(s)  (\Pi_0(s))_{cd}    \tau_a\tau_b  \sqrt{ 1 + s^2| \nabla \tau |^2 }    \ d v_{\sigma_0}
= \int_\Sigma s  \Pi_0 ( \nabla \tau, \nabla \tau) d v_\Sigma +F_4
\end{split}
\ee
where
\be \label{est-e4-2}
|F_4|\le C_6 s ||\tau||_{C^{3,\alpha}}^2\int_\Sigma |\nabla \tau|^2 d v_{\Sigma}
\ee
for some constant $C_6$ depending only on $\sigma$.
By \eqref{assumption-e2}, \eqref{1stvar-e1} and \eqref{est-e1}--\eqref{est-e4-2},
we have
\be \label{est-1stvar}
 \begin{split}
  \frac{d}{ds}\Ewy(s)   =  & \
  s \int_\Sigma  \left[ \frac{ (  \Delta  \tau)^2 }{ H   } + ( H_0 - H ) | \nabla \tau |^2
 -  \int_\Sigma \Pi_0 ( \nabla \tau, \nabla \tau )  \right] d v_\Sigma  \\
 & \ +F_1+F_2+F_3+F_4\\
 \ge &s(\beta-C_7 ||\tau||_{C^{3,\alpha}}^2)\int_\Sigma (\Delta\tau)^2dv_{\Sigma}
  \end{split}
  \ee
  for some constant $C_7 $ depending only on $\sigma$, where in the last step we have also
  used the fact (see \eqref{gradandlap} below) that
  \be
 \lambda_1  \int_\Sigma|\nabla\tau|^2dv_{\Sigma}\le \int_\Sigma\lf(\Delta\tau\ri)^2dv_{\Sigma}
 \ee
 with $\lambda_1 $ being the first nonzero eigenvalue of the Laplacian of $ \sigma$.
Hence,  if $\e$ is chosen such that $0<\e^2<\e_1$ and
$\beta-C_7 \e^2>\frac12\beta$, then we have
 \be \label{dEest}
 \frac{d}{ds}\Ewy(s) \ge \frac12 s \beta \int_\Sigma ( \Delta \tau)^2 d v_\Sigma
 \ee
 for any $0 \le s \le1$ and for any  smooth  function $ \tau $
 with $ ||\tau||_{C^{3,\alpha}}<\e$. In particular, this implies
 \be
 \Ewy (\Sigma, \tau) \ge \Ewy(\Sigma, 0) + \frac{\beta}{4} || \Delta \tau ||_{L^2}^2 .
 \ee
Theorem \ref{minimum-t1} is proved.
\end{proof}

The following corollary gives a simple condition in terms of $ \sigma$ and $ H$ that
guarantees \eqref{assumption-e2} in Theorem \ref{Mainresult}.

\begin{cor} \label{cor-eigenvalue}
Let $ N$, $ \Omega$, $ \Sigma$, $ \sigma $, $ H$, $H_0$ and $ \Pi_0$ be given as in Theorem \ref{minimum-t1}.
Suppose the first non-zero eigenvalue $\lambda_1$ of the Laplacian of $\sigma$ satisfies:
      \be \label{1st-eigenvalue-e1}
      \lambda_1>H^{\max}\lf(H^{\max}-\Pi_0^{\min}\ri)
      \ee
      where $H^{\max}=\max_{\Sigma}H$ and $\Pi_0^{\min}$ is the minimum of all the eigenvalues of $\Pi_0$ on     $(\Sigma, \sigma)$.
Then condition \eqref{assumption-e2} holds, hence
 $\mby(\Sigma,\Omega)$ strictly locally minimizes $\Ewy(\Sigma,\cdot)$
\end{cor}

\begin{proof} By Theorem \ref{minimum-t1}, it suffices to show that there exists a constant $ \beta > 0 $ such that \eqref{assumption-e2}  holds for all $ \eta \in W^{2,2}(\Sigma)$.

First, we note that
\be \label{gradandlap}
 \lambda_1  \int_\Sigma|\nabla \eta|^2dv_{\Sigma}\le \int_\Sigma\lf(\Delta \eta \ri)^2dv_{\Sigma} , \
 \ \forall \ \eta \in W^{2,2} (\Sigma).
\ee
To verify this, it suffices to assume $ \int_\Sigma \eta d v_\Sigma  = 0$.
For such an $ \eta$, we have
\be
\begin{split}
 \int_{\Sigma}|\nabla\eta|^2 dv_{\Sigma}
=& -\int_{\Sigma}\eta\Delta\eta dv_{\Sigma}\\
\le & \lf(\int_{\Sigma}\eta^2 dv_{\Sigma}\ri)^\frac12\lf(\int_{\Sigma}\lf(\Delta\eta\ri)^2 dv_{\Sigma}\ri)^\frac12\\
\le & \lf(\lambda_1^{-1}\int_{\Sigma}|\nabla\eta|^2 dv_{\Sigma}\ri)^\frac12\lf(\int_{\Sigma}\lf(\Delta\eta\ri)^2 dv_{\Sigma}\ri)^\frac12
\end{split}
\ee
which implies \eqref{gradandlap}.

Now suppose (i) holds. By the definition of $ \Pi_0^{min} $, we have
\bee
H_0|\nabla\eta|^2-\Pi_0(\nabla\eta,\nabla\eta)\ge \Pi_0^{\min}|\nabla\eta|^2 .
\eee
Therefore,
\be
\begin{split}
& \ \int_{\Sigma} \lf[ \frac{ ( \Delta \eta )^2 }{ H   } + ( H_0 -   H  ) | \nabla \eta  |^2 - \Pi_0 ( \nabla \eta ,\nabla \eta)\ri]dv_{\Sigma}\\
 \ge &  \ \int_{\Sigma}\lf[ \frac{(\Delta\eta)^2}{H^{\max}}+\lf(\Pi_0^{\min}-H^{\max}\ri)|\nabla\eta|^2\ri]dv_{\Sigma}\\
= & \ \frac1{H^{\max}}\int_{\Sigma}\lf[(\Delta\eta)^2-(\lambda_1-\delta) |\nabla\eta|^2 \ri] dv_{\Sigma}\\
\ge & \  \frac{\delta_1}{H^{\max}}\int_{\Sigma}(\Delta\eta)^2dv_{\Sigma}
\end{split}
\ee
where $\delta=\lambda_1-H^{\max}\lf(H^{\max}-\Pi_0^{\min}\ri)>0$, and $\delta_1=\min\{1, \delta/\lambda_1\} $ which is positive. Hence, \eqref{assumption-e2} is satisfied with $ \beta =   {\delta_1} / {H^{\max}} $.

\end{proof}

We leave it to the interested readers to verify that those surfaces $ \Sigma $ in  (i) and (ii) provided in Section \ref{Introduction} also satisfy the condition \eqref{1st-eigenvalue-e1} in the above Corollary.

\section{Strict positivity of the second variation}\label{moreonHcondition}

We investigate the condition \eqref{assumption-e2} in this section.
Our main result is the following theorem:

\begin{thm}\label{Mainresult-1}
Let $ \Omega $ be a three dimensional, compact Riemannian manifold with boundary
$ \p \Omega$. Suppose each component of $ \p \Omega$ has positive mean curvature.
Let $ \Sigma$ be a component of $ \p \Omega$. Suppose the induced metric $ \sigma $ on
$ \Sigma$ has positive Gaussian curvature and
\be \label{Hcondition-1}
 H \le H_0
 \ee
where $ H$ is   the mean curvature of $ \Sigma $ in $ \Omega$ and $ H_0$ is
the mean curvature of $ \Sigma$ when isometrically embedded in $ \R^3$.
If $ \Omega $ is not isometric to a domain in $ \R^3$, then
there exists a constant $\beta>0$ such that
\be\label{assumption-e-2-1}
 \int_\Sigma \lf[ \frac{ ( \Delta \eta )^2 }{ H  } + ( H_0 - H ) | \nabla \eta  |^2 - \Pi_0 ( \nabla \eta ,\nabla \eta)\ri]dv_{\Sigma}\ge \beta\int_\Sigma (\Delta\eta)^2dv_{\Sigma}
\ee
for all $\eta\in W^{2,2}(\Sigma)$. Here $\Pi_0$ is  the second fundamental form of $(\Sigma,\sigma)$ when isometrically embedded in $\R^3$.
\end{thm}

We divide the  proof of Theorem \ref{Mainresult-1} into a few steps. First,
we consider the left side of \eqref{assumption-e-2-1} in the case that $ \Omega$ is
indeed a domain in $ \R^3$. That leads to  a  result concerning manifolds  with nonnegative Ricci curvature.

\begin{prop} \label{Reilly-t}
Let $ (\Omega, g)$ be a compact Riemannian manifold of dimension $ n \ge 3 $.
Suppose $ \Omega $ has smooth boundary $ \p \Omega $ (possibly disconnected)
which has positive mean curvature
$ H $.
If $ g $ has nonnegative Ricci curvature, then
\be \label{secondvar-R4}
\int_{\p \Omega} \lf[ \frac{ ( \Delta \eta )^2 }{ H }   - \Pi ( \nabla \eta ,\nabla \eta )\ri]
 d v_{_{\p \Omega} } \ge 0
\ee
for any smooth function $  \eta $ on $\p \Omega$. Here $ \Pi $ is the second fundamental form
of $ \p \Omega $ in $(\Omega, g)$,
$ \nabla $ and $ \Delta $ are the
gradient and Laplacian on $ \p \Omega$ and $ d v_{_{\p \Omega} }$ is the volume form on $ \p \Omega $.

 Moreover, equality in \eqref{secondvar-R4} holds for some $ \eta $ if and only if $ \eta $
 is the boundary value of some smooth function $ u $ which satisfies
 $ \nabla^2_\Omega u = 0 $ and $ \Ric( \nabla_{_\Omega} u, \nabla_{_\Omega} u ) = 0 $ on $ \Omega$. Here $ \nabla^2_\Omega $ and $ \nabla_{_\Omega} $ denote the
 Hessian and the gradient on $(\Omega, g)$.
\end{prop}

\begin{proof}
Given a smooth function $ \eta $ on $ \p \Omega $, let $ u $ be the harmonic function
on $(\Omega, g)$ such that $ u = \eta $ on $ \p \Omega$.
By the Reilly formula \cite[Equation (14)]{Reilly77} (see also \cite[Theorem 8.1]{PeterLi93}),  we have
\be \label{Reilly}
- \int_{\p \Omega} \lf[ \Pi( \nabla u, \nabla u) + 2 \frac{\p u}{\p \nu} \Delta u  + H \lf( \frac{\p u}{\p \nu } \ri)^2 \ri ]
= \int_\Omega | \nabla^2_{_\Omega} u |^2 + \Ric( \nabla_{_\Omega} u, \nabla_{_\Omega} u)
\ee
where $ \Ric (\cdot, \cdot)$ is the Ricci curvature of $g$. Here we omit
the corresponding volume form in each integral.

Since $ \Ric( \cdot, \cdot ) \ge 0 $, \eqref{Reilly} implies
\be \label{Reilly-CS}
\begin{split}
\int_\Sigma \Pi( \nabla u, \nabla u)  \le  & \ \int_\Sigma - 2 \frac{\p u}{\p \nu} \Delta u  - H \lf( \frac{\p u}{\p \nu } \ri)^2 \\
\le & \ \int_\Sigma  \frac{ ( \Delta \eta )^2 }{ H }
\end{split}
\ee
by the Cauchy-Schwarz inequality. Hence \eqref{secondvar-R4} is proved.

Now suppose the equality in \eqref{secondvar-R4} holds, then the equalities in \eqref{Reilly-CS}
must hold. In particular, we have
\be
\int_\Omega | \nabla^2_{_\Omega} u |^2 + \Ric( \nabla_{_\Omega} u, \nabla_{_\Omega} u) = 0,
\ee
which shows $ \nabla^2_{_\Omega} u = 0 $ and $ \Ric( \nabla_{_\Omega} u, \nabla_{_\Omega} u ) = 0 $ on $ \Omega$.
On the other hand, if $ \nabla^2_{_\Omega} u = 0 $ on $ \Omega $, then
\be
\Delta u + H \frac{ \p u }{ \p \nu } = 0 \ \ \mathrm{on} \ \Sigma
\ee
which shows the second equality in \eqref{Reilly-CS} must hold.
If in addition $ \Ric( \nabla_{_\Omega} u, \nabla_{_\Omega} u) = 0 $,
then the first equality in \eqref{Reilly-CS} holds as well.
Proposition \ref{Reilly-t} is proved.
\end{proof}

\begin{rem}
We thank Michael Eichmair who brings Reilly's formula \eqref{Reilly} to our attention.
\eqref{Reilly} was derived by integrating the Bochner formula and expressing the boundary term
$ \frac12 \int_\Sigma \frac{\p  }{ \p \nu} | \nabla_{_\Omega} u |^2 $ as the left  side of
\eqref{Reilly}. In particular, Proposition \ref{Reilly-t} remains valid under the general assumption that
the mean curvature $ H $ does not change sign on each  component of $ \p \Omega$.
\end{rem}

Specializing Proposition \ref{Reilly-t} to domains in $ \R^n $, we have

\begin{cor} \label{2ndvarinrn}
Let $ \Omega $ be a bounded domain in $  \R^n $ $( n \ge 3 )$
with a smooth connected boundary $ \Sigma$. Suppose $ \Sigma $ has positive mean curvature $ H_0 $.
Let $ \Pi_0 $ be the second fundamental form
of $ \Sigma $ in $ \R^n$.
Then
\be \label{secondvar-R5}
\int_\Sigma \lf[ \frac{ ( \Delta \eta )^2 }{ H_0 }   - \Pi_0 ( \nabla \eta ,\nabla \eta )\ri]dv_{\Sigma}\ge 0
\ee
for any smooth function $  \eta $ on $\Sigma$, where  $ \nabla $ and $ \Delta $ are the
gradient and the Laplacian on $ \Sigma$ and $ d v_\Sigma $ is the volume form on $ \Sigma $.
Moreover, equality in \eqref{secondvar-R5} holds for some $ \eta $ if and only if
 $ \eta $ is the restriction of a linear function to $ \Sigma$, i.e.
 $ \eta = a_0 + \sum_{i=1}^n a_i x^i $ for some constants $a_0, a_1, \ldots, a_n $.
\end{cor}

\begin{rem}
When $ n = 3$ and $ \Sigma $ is a strictly convex surface in $ \R^3 $, the inequality
\eqref{secondvar-R5}
can also be seen by considering the second variation of $ \Ewy(\Sigma, \cdot)$
for  $ \Sigma \subset \R^3=\{(x,0)\in \R^{3,1}\}$.
In fact, by \cite[Theorem A]{WangYau08},
$ \Ewy ( \Sigma, \tau ) \ge 0 $ for any admissible function $ \tau $. Since $ \Sigma $
has positive Gaussian curvature, $ \tau $ is admissible if $ || \tau ||_{C^{3, \alpha} }$ is sufficiently small \cite[Remark 1.1]{WangYau08}.
Therefore, $ \Ewy (\Sigma, \tau) \ge 0$ for any such $ \tau$. On the other hand, it is obvious that $ \Ewy( \Sigma, 0 ) = 0$.
Hence, \eqref{secondvar-R5} follows  from \eqref{2ndvarformula}.
\end{rem}

Next, we derive an estimate of the left side of \eqref{assumption-e-2-1} for those $ \eta $
which are restriction of linear functions in $ \R^3$ to $ \Sigma$.

\begin{prop} \label{linearcase}
Let $ \Omega$ be a three dimensional Riemannian manifold.
Let $ \Sigma \subset \Omega $ be an embedded closed $2$-surface that is diffeomorphic
to a sphere. Suppose the induced metric $ \sigma$ on $ \Sigma $ has positive Gaussian curvature.
Let
$$ X = (X^1, X^2, X^3): \Sigma \hookrightarrow \R^{3} $$
be an isometric embedding of $ (\Sigma, \sigma) $ into $ \R^3$.
Given any constant $ a_0 $ and any constant unit vector $ a = (a_1, a_2, a_3) \in \R^3$, let
$ \eta = a_0 +  \sum_{i=1}^3 a_i X^i$,
then
\be\label{assumption-e-2-2}
 \int_\Sigma \lf[ \frac{ ( \Delta \eta )^2 }{ H  } + ( H_0 - H ) | \nabla \eta  |^2 - \Pi_0 ( \nabla \eta ,\nabla \eta)\ri]dv_{\Sigma} \ge 8 \pi \mby (\Sigma, \Omega)
\ee
where  $ H$ is   the mean curvature of $ \Sigma $ in $ \Omega$, $ H_0$ and $ \Pi_0$ are
the mean curvature and the second fundamental form of $ \Sigma$ when isometrically embedded in $ \R^3$.
\end{prop}

\begin{proof}
For such an $ \eta$, Corollary \ref{2ndvarinrn} implies
\be
\int_\Sigma \lf[ \frac{ ( \Delta \eta )^2 }{ H_0 }   - \Pi_0 ( \nabla \eta ,\nabla \eta )\ri]dv_{\Sigma} = 0 .
\ee
Direct calculation shows
\be
( \Delta \eta )^2  = ( a \cdot \vec{H}_0 )^2 \ \
\mathrm{and}
\ \
| \nabla \eta |^2 = 1 - ( a \cdot \vec{H}_0)^2  H_0^{-2}
\ee
where $ \vec{H}_0$ is the mean curvature vector of $ \Sigma $ when isometrically embedded in $ \R^3$.
Therefore,
\be
\begin{split}
\ & \ \int_\Sigma \lf[ \frac{ ( \Delta \eta )^2 }{ H  } + ( H_0 - H ) | \nabla \eta  |^2 - \Pi_0 ( \nabla \eta ,\nabla \eta)\ri]dv_{\Sigma}  \\
= & \ \int_\Sigma ( a \cdot \vec{H}_0 )^2 \lf( \frac{1}{H} - \frac{2}{H_0} + \frac{H}{H_0^2} \ri)
+ ( H_0 - H) d v_{\Sigma} \\
\ge & \ \int_\Sigma ( H_0 - H ) d v_\Sigma .
\end{split}
\ee
\end{proof}

We are now ready to prove Theorem \ref{Mainresult-1}.

\vspace{.2cm}

\noindent {\em Proof of Theorem \ref{Mainresult-1}}.
For convenience, we omit writing the volume form in each integral.
Note that
\be \label{eq-rewriting}
\int_\Sigma \lf[ \frac{ ( \Delta \eta )^2 }{ H  } + ( H_0 - H ) | \nabla \eta  |^2 - \Pi_0 ( \nabla \eta ,\nabla \eta)\ri]
  = I_1 (\eta, \eta) + I_2 ( \eta, \eta)
\ee
where
$$ I_1 (\eta, \eta) =  \int_\Sigma \lf[ \frac{ ( \Delta \eta )^2 }{ H  } - \frac{ (\Delta \eta )^2}{H_0} + ( H_0 - H ) | \nabla \eta  |^2 \ri]  $$
and
$$ I_2 (\eta, \eta) = \int_\Sigma \lf[ \frac{ ( \Delta \eta )^2 }{ H_0 } - \Pi_0 ( \nabla \eta , \nabla \eta) \ri]   .$$
By Corollary \ref{2ndvarinrn}, we know $ I_2 (\eta, \eta) \ge 0 $.
By  the assumption \eqref{Hcondition-1}, we have
$ I_1 (\eta, \eta) \ge 0 $.
Therefore,
\be \label{eq-nonnegative}
\int_\Sigma \lf[ \frac{ ( \Delta \eta )^2 }{ H  } + ( H_0 - H ) | \nabla \eta  |^2 - \Pi_0 ( \nabla \eta ,\nabla \eta)\ri]  \ge 0 .
\ee
To prove \eqref{assumption-e-2-1}, we  argue by contradiction.

Suppose \eqref{assumption-e-2-1} is not true, then there exists a sequence of functions $ \{ \eta_k \} \subset W^{2,2}(\Sigma)$ with
\be \label{l2eq1}
\int_\Sigma \eta_k = 0   \
 \mathrm{and} \
 \int_{\Sigma}\eta_k^2 =1
\ee
such that
\be \label{secondvar-R5-2}
\int_\Sigma \lf[ \frac{ ( \Delta \eta_k )^2 }{ H  } + ( H_0 - H ) | \nabla \eta_k |^2 - \Pi_0 ( \nabla \eta_k ,\nabla \eta_k)\ri]
\le \frac1k \int_{\Sigma}(\Delta\eta_k)^2 .
\ee
By the interpolation inequality for Sobolev functions, we have
\be \label{eq-2}
\begin{split}
\int_{\Sigma}
\frac{ (\Delta \eta_k )^2}{ H} \le & \ \int_\Sigma \lf[ ( H - H_0) | \nabla \eta_k |^2 + \Pi_0 ( \nabla \eta_k , \nabla \eta_k ) \ri]   +   \frac{1}{k}  \int_\Sigma ( \Delta \eta_k)^2 \    \\
\le & \ C_1 + \int_{\Sigma}
\frac{ (\Delta \eta_k )^2}{ 2 H}  +   \frac{1}{k}  \int_\Sigma ( \Delta \eta_k)^2 \   .
\end{split}
\ee
Here and below, $\{ C_1, C_2, \ldots\}$ denote positive constants independent on $ k $.
It follows from \eqref{eq-2}  that
\be
 || \Delta \eta_k ||_{L^2 (\Sigma) }  \le C_2 .
\ee
By \eqref{l2eq1} and the usual $L^p$ estimate, we then  have
\be \label{w22bded}
  ||  \eta_k ||_{W^{2,2} (\Sigma) }  \le C_3 .
  \ee
This implies that  there exists a function $ \eta \in W^{2, 2} (\Sigma)$ such that
\begin{itemize}
\item[a)] $ \eta_k $ converges weakly to $ \eta $ in $ W^{2,2} (\Sigma) $.
\item[b)] $ \eta_k $ converges strongly to $ \eta $ in $ W^{1,2} (\Sigma)$.
\end{itemize}
By \eqref{w22bded} and a), b),  one also easily verifies that
\begin{itemize}
\item[c)] $ \Delta \eta_k $ converges to $ \Delta \eta $ weakly in $ L^2 (\Sigma)$.
\end{itemize}
Moreover, by \eqref{l2eq1} and b),  $ \eta $ satisfies
 \be \label{eq-normalization}
 \int_\Sigma \eta  = 0 \  \mathrm{and}  \  \int_\Sigma \eta^2   = 1 .
\ee
We now claim that
\be \label{eq-3}
 \int_{\Sigma}
\lf[ \frac{ (\Delta \eta )^2}{ H} + ( H_0 - H) | \nabla \eta |^2 - \Pi_0 ( \nabla \eta , \nabla \eta ) \ri]    = 0 .
\ee
To see this, we replace $ \eta $ by $ \eta - \eta_k $ in \eqref{eq-nonnegative} to obtain
\be \label{eq-etabyetaetak}
\int_{\Sigma}
\lf[ \frac{ (\Delta (\eta - \eta_k ) )^2}{ H} + ( H_0 - H) | \nabla ( \eta - \eta_k ) |^2 - \Pi_0 ( \nabla (\eta - \eta_k ), \nabla ( \eta - \eta_k)  ) \ri]    \ge 0  .
\ee
It follows from \eqref{secondvar-R5-2} and \eqref{eq-etabyetaetak} that
\be \label{eq-etabyetaetak-1}
\begin{split}
\frac{1}{k} \int_\Sigma ( \Delta \eta_k )^2 d v_\Sigma \ge & \  \int_{\Sigma}
  \lf[ \frac{ (\Delta \eta_k )^2}{ H} + ( H_0 - H) | \nabla \eta_k |^2 - \Pi_0 ( \nabla \eta_k , \nabla \eta_k ) \ri] \\
\ge & \  \int_\Sigma \frac{ 2 \Delta \eta_k \cdot \Delta \eta -   (\Delta \eta )^2}{ H} + ( H_0 - H) ( 2 \nabla \eta_k \cdot \nabla \eta - | \nabla \eta |^2)  \\
& + \int_\Sigma - 2 \Pi_0 ( \nabla \eta_k , \nabla \eta) + \Pi_0 ( \nabla \eta, \nabla \eta)  .
\end{split}
\ee
Letting $ k \rightarrow \infty$, by \eqref{w22bded}, a), b), c) and \eqref{eq-etabyetaetak-1} we have
\be
0 \ge \int_{\Sigma}
  \lf[ \frac{ (\Delta \eta )^2}{ H} + ( H_0 - H) | \nabla \eta |^2 - \Pi_0 ( \nabla \eta , \nabla \eta ) \ri]  .
\ee
This, together with \eqref{eq-nonnegative}, shows that
\be \label{eq-zero}
 \int_{\Sigma}
  \lf[ \frac{ (\Delta \eta )^2}{ H} + ( H_0 - H) | \nabla \eta |^2 - \Pi_0 ( \nabla \eta , \nabla \eta ) \ri]  = 0 .
\ee

Next, we claim that $ \eta $ must be the restriction of a linear function
on $ \Sigma$. Here we identify $ \Sigma $ with its image in $ \R^3$ under the isometric embedding.
To see this, first we note that $ \eta $ is a smooth function on $ \Sigma$. That is because,
by \eqref{eq-nonnegative} and \eqref{eq-zero}, $ \eta $ is a minimizer of the functional
$$ \mathcal{F} ( f ) =  \lf[ \frac{ (\Delta f )^2}{ H} + ( H_0 - H) | \nabla f |^2 - \Pi_0 ( \nabla f , \nabla f ) \ri]  $$
on $ W^{2,2} (\Sigma) $. Hence, $ \eta $ is a weak solution to the Euler-Lagrange equation
\be \label{eq-varPDE}
\Delta \lf( \frac{ \Delta \eta }{ H } \ri) - \div \lf( ( H_0 - H ) \nabla \eta \ri) + \div ( \Pi_0 ( \cdot, \nabla \eta  ) ) = 0 .
\ee
Since the coefficients of \eqref{eq-varPDE} are assumed to be smooth, we know
$ \eta $ is a smooth function  by the standard elliptic regularity theory.
Second, by \eqref{eq-rewriting}, we have
\be
0 = I_1 (\eta, \eta ) + I _2 (\eta, \eta) .
\ee
Since $ I_1 (\eta, \eta) \ge 0 $ and $ I_2 (\eta, \eta) \ge 0$, we know  $ I_2 ( \eta, \eta ) = 0 $.
By Corollary \ref{2ndvarinrn}, we conclude that
\be
\eta = a_0 + \sum_{i=1}^3 a_i x^i
\ee
for some constants $ a_0, a_1, a_2, a_3$.
By  \eqref{eq-normalization} we further know that
$ \eta $ is not a constant, hence
$( a_1, a_2, a_3)  \neq (0, 0, 0)$.

For such an $ \eta$, Proposition \ref{linearcase} shows
\be
\int_\Sigma \lf[ \frac{ ( \Delta \eta )^2 }{ H  } + ( H_0 - H ) | \nabla \eta  |^2 - \Pi_0 ( \nabla \eta ,\nabla \eta)\ri]
 \ge 8 \pi \mby (\Sigma, \Omega).
\ee
Therefore, by \eqref{eq-zero} we have
\be
0 \ge 8 \pi \mby (\Sigma, \Omega) =  \int_\Sigma ( H_0 - H )  .
\ee
Since it is assumed $ H_0 \ge H $ on $ \Sigma$, we conclude  that $ H_0 =  H $
everywhere on $ \Sigma$.

To finish the proof, we apply the positive mass theorem to draw a contradiction.
Let $ N \subset \R^3$ be the exterior region of $ \Sigma$. We attach $ N$  to the
compact manifold $ \Omega$ along $ \Sigma $ to get a Riemannian manifold $ M$.
The metric $ g_M $ on $M$ has the feature that, though it may not be smooth across
$ \Sigma$, the mean curvatures of $ \Sigma $ from its both sides in $ M$ agree.
We have the following two cases:

\begin{itemize}

\item   When $ \p \Omega $ has only one component, i.e. $ \Sigma = \p \Omega$, we can apply Theorem 3.1 in \cite{ShiTam02} (or Theorem 2 in \cite{Miao02})  directly to  conclude that $ \Omega $  must be isometric to a domain in $ \R^3$.
This is a contradiction to the assumption on $ \Omega$.

\vspace{.1cm}

\item
When $ \p \Omega $ has more than one components, $ M$ has a nonempty boundary
$ \p M = \p \Omega \setminus \Sigma$, which by assumption has positive mean curvature
(i.e. its mean curvature vector points inside $M$).
In this case, one can modify the proof of Theorem 3.1 in \cite{ShiTam02}
 to show that $ \Omega $ still must be isometric  to a domain in $ \R^3$. Or one can proceed as in
\cite[Section 3.2]{Miao-LRPI} to draw a contradiction as follows:
by minimizing  area among surfaces in $ \Omega $ that are homologous to $ \Sigma$, we know
 there exists a  closed  minimal surface $ \Sigma_H $ in $ \Omega$  having the property that there are no other closed minimal surface lying inside the region $ \tilde{\Omega} $ bounded
 by $ \Sigma $ and $ \Sigma_H$.
By directly applying Lemma 2, 3, 4 in \cite{Miao-LRPI} and the Riemannian Penrose
inequality \cite{Bray, Huisken-Ilmanen}, we have
$$
\mathrm{the \ mass \ of \ } g_M \ge \sqrt{ \frac{ | \Sigma_H | }{ 16 \pi} }  > 0.
$$
This contradicts the fact that $ M $ outside $ \Sigma $ is the exterior Euclidean region $ N$, which has
zero mass.
\end{itemize}

We conclude that \eqref{assumption-e-2-1} is true.  Hence, Theorem \ref{Mainresult-1} is  proved. \stop

\vspace{.2cm}

Part (1) of Theorem \ref{Mainresult} now follows directly from Theorem \ref{minimum-t1} and Theorem \ref{Mainresult-1}.

\section{Second fundamental form of the isometric embedding
} \label{secondfform}

The rest of this paper is devoted to study of {\bf Question 2}. As mentioned in the introduction,
in order to apply the {IFT},  we want to verify that the map, which
sends a metric $ \sigma $ (on the two-sphere $S^2$) of positive Gaussian curvature to the second fundamental form of the isometric embedding of $(S^2, \sigma)$ into $ \R^3$,
is a $ C^1 $ map between appropriate functional spaces.
To do so, we follow closely the original work of Nirenberg \cite{Nirenberg}.

First, we fix some notations. Let $ \Sigma = S^2$.
Given an integer $ k \geq 2 $ and a positive number
$0< \alpha < 1 $, let
$$
\begin{array}{lll}
 \E^{k, \alpha} & = & \mathrm{the \ space \ of } \ C^{k,\alpha} \ \mathrm{embeddings \ of} \ \Sigma
\ \mathrm{into}  \ \R^3 \\
 \X^{k, \alpha} & =  & \mathrm{the \ space \ of \ } C^{k,\alpha} \ \R^3
 \mathrm{-valued \ vector \ functions \ on \ } \Sigma \\
 \Sp^{k, \alpha} & = & \mathrm{the \ space \ of \ } C^{k,\alpha} \  \mathrm{symmetric \ (0,2) \ tensors \ on}
 \ \Sigma \\
 \M^{k, \alpha} & = & \mathrm{the \ space \ of \ } C^{k,\alpha} \  \mathrm{Riemannian \  \ metrics \ on}
 \ \Sigma \\
  \M_+^{k, \alpha}&=&\mathrm{open\ subset \ of \M^{k, \alpha}\ \textrm{with positive Gaussian curvature}} .
\end{array}
$$
 By the results in \cite{Nirenberg}, for $k\ge 4$ and $\sigma\in \M_+^{k, \alpha}$,
 there is an isometric embedding $X(\sigma)$ of $(\Sigma,\sigma)$ into $\R^3$ which is unique up to an isometry of $\R^3$. Also, $X(\sigma)$ is necessarily in $ \E^{k,\alpha}$ by \cite{Nirenberg}. Hence the following map is well-defined:
\be \label{dofmapF}
\mathcal{F}:\M_+^{k, \alpha}\to \Sp^{k-2,\alpha}\subset \Sp^{k-3,\alpha}
\ee
 where
$\mathcal{F}(\sigma)=\mathbb{II}(X(\sigma))$  is
the second fundamental form of $ X (\sigma) ( \Sigma ) $ (pulled back via $X(\sigma) $ and viewed as an element
in $ \Sp^{k - 2, \alpha} $). We want to study the smoothness of $\mathcal{F}$.

Given $\sigma\in \M_+^{k, \alpha}$, $k\ge 4$  and
let $X=X(\sigma)\in \E^{k,\alpha}$  be an isometric embedding of
$(\Sigma,\sigma)$. Let $ \{(u, v) \}$ denote a fixed coordinate chart on $ \Sigma$, let $ X_u$, $ X_v $ denote the partial derivative of $ X$ with respect to $ u$, $v$, and let $X_3=X_u\wedge X_v/|X_u\wedge X_v|$  be the unit normal.   The coefficients of the first and the second
fundamental forms of $ X  $ are denoted by $E, F, G$ and $L, M, N$
respectively. Let $\Delta=\sqrt{EG-F^2}$ and let $K$, $H$ be
the Gaussian curvature, the mean curvature of $ X(\Sigma) $ which are both positive. Let
 $$
  \begin{pmatrix}
      l & m \\
      m & n \\
    \end{pmatrix}=\begin{pmatrix}
                    L & M \\
                    M & N \\
                  \end{pmatrix}^{-1}
$$
and
$$
\begin{pmatrix}
      A & B \\
    C & D \\
    \end{pmatrix}=\begin{pmatrix}
                    E & F \\
                    F & G \\
                  \end{pmatrix}^{-1}\begin{pmatrix}
                    L & M \\
                    M & N \\
                  \end{pmatrix}.
$$
Note that
\begin{equation*}
     \begin{split}
        (X_3)_u=& AX_u+ B X_v \\
         (X_3)_v=& CX_u+ D X_v
     \end{split}
\end{equation*}

By \cite[Section 6-8]{Nirenberg}, given any $\rho\in\Sp^{r, \alpha}$ and
$r\ge 2$, there exists a uniquely determined  $Y = \Phi ( \sigma, \rho) \in \X^{s,\alpha}$ (which also depends on $ X$), where $s=\min\{k-1,r\}$, such that
$ Y $ is a solution of
\be \label{eq-truelinear}
2dX\cdot dY=\rho
\ee
and $ Y $  vanishes at a fixed point on $\Sigma$.
Recall from \cite{Nirenberg} that $ Y $ is constructed in the following way:


\vspace{.2cm}

\noindent {\sc Step 1}: Let $\phi$ be the unique solution of
\be\label{phi-e1}
 \mathcal{L}(\phi_u,\phi_v)+H\phi=\mathcal{L}(c_1,c_2)
 -T
\ee
which is $L^2$-orthogonal to the kernel of $\mathcal{L}(\phi_u,\phi_v)+H\phi$ which
 is spanned by the coordinates functions of $X_3$. Here
 \be\label{L-e1}
 \begin{split}
  \mathcal{L}(q_1,q_2)=&\frac{1}{\Delta}\lf(\frac
  N{K\Delta}q_1-\frac{M}{K\Delta}q_2\ri)_u-
\frac{1}{\Delta}\lf(\frac
 M{K\Delta}q_1-\frac{L}{K\Delta}q_2\ri)_v\\
 \end{split}
  \ee
  \be \label{c-e1}
  c_1=\frac{1}{\Delta}\lf(\rho_{12;u}-\rho_{11;v}\ri),\qquad c_2=\frac{1}{\Delta}\lf(\rho_{22;u}-\rho_{21;v}\ri)
  \ee
  \be \label{T-e1}
  T=\frac1{\Delta}\lf(C\rho_{11}+(D-A)\rho_{12}-B\rho_{22}\ri),
  \ee
where $\rho_{ij;u}$ etc. are the covariant derivatives of $\rho$ on $(\Sigma, \sigma)$.

Denote $\phi=\Psi(\sigma,\rho)$. Note that $\Psi$ is linear in $\rho$.

\vspace{.2cm}

\noindent {\sc Step 2}: $Y=\Phi(\sigma,\rho)$ is obtained by integrating:
\be\label{Y-e1}
\begin{split}
  Y_u=&\frac{1}{2\Delta^2}\lf(\rho_{11}G-\rho_{12}F\ri)X_u+
  \frac{1}{2\Delta^2}\lf(\rho_{12}E-\rho_{11}F\ri)X_v
  +\frac1{2\Delta}\lf(EX_v-FX_u\ri)\phi+X_3p_1
 \\
  Y_v=&\frac{1}{2\Delta^2}\lf(\rho_{12}G-\rho_{22}F\ri)X_u+
  \frac{1}{2\Delta^2}\lf(\rho_{22}E-\rho_{12}F\ri)X_v
  +\frac1{2\Delta}\lf(FX_v-GX_u\ri)\phi+X_3p_2.
  \end{split}
 \ee
where
\be\label{p-e1}
\begin{pmatrix}
    p_1 \\
    p_2 \\
  \end{pmatrix}=\frac{\Delta}2\begin{pmatrix}
                   m & n \\
l & -m \\
\end{pmatrix}
\begin{pmatrix}
   \phi_u- c_1 \\
    \phi_v-c_2 \\
  \end{pmatrix} .
  \ee
In particular, $ \Phi $ is linear in $\rho$.
By (6.6) in \cite{Nirenberg}, $ \phi$ and $ \Phi$ are also  related by
\be \label{phiandPhi}
\phi (u, v) = \frac{1}{\Delta} \lf( X_v \cdot \Phi_u - X_u \cdot \Phi_v \ri) .
\ee

The following $ C^0 $ estimate of $ \phi$ was proved in  \cite[Lemma 5.2]{MiaoShiTam}.

\begin{lma}\label{phi-0-l1}
Let  $ \sigma_0 \in \M_+^{5, \alpha} $.   There exists positive numbers $ \epsilon$
  and $ C$, depending only on $ \sigma_0$,  such that if $\sigma\in
\M_+^{4 , \alpha}$ and $||\sigma-\sigma_0||_{C^{2,\alpha}}<\epsilon$, then for any $\rho\in \Sp^{r,\alpha}$, $r\ge 2$,
$$
||\phi||_{C^0}\le C||\rho||_{C^{1,\alpha}}
$$
where $\phi=\Psi(\sigma,\rho)$.
\end{lma}

 Let $\sigma\in \M_+^{k, \alpha}$, $k\ge 4$ and let $X = X(\sigma)$ be a given isometric embedding of $(\Sigma,\sigma)$. By \cite[Section 5]{Nirenberg}, for any $\tau\in \M_+^{k, \alpha}$ which is close to $\sigma$ in the $C^{2,\alpha}$ norm, there exists an  isometric embedding of $(\Sigma,\tau)$ in the form of $X+Y$ where $Y$ is obtained as follows: Let $Y_0=0$ and $Y_m=\Phi(\sigma,\rho_{m-1})$, where $\rho_{m-1}=\tau-\sigma -  (d Y_{m-1})^2$, then $ \{ Y_m\}$ converges to $Y$ in the $C^{2,\alpha}$ norm such that $Y $ satisfies:
 \be \label{eq-isoembedding}
 2dX\cdot dY=\tau-\sigma - (dY)^2.
 \ee
 Let us denote this particular solution $Y$ to \eqref{eq-isoembedding}
  by $Y(\sigma,\tau)$. Since both $ X (\sigma) $ and $ X(\sigma)  + Y(\sigma, \tau) $ are in  $ \E^{k, \alpha}$,
we know $ Y(\sigma, \tau)$  is of $ C^{k, \alpha}$.

In \cite[Lemma 5.3]{MiaoShiTam}, the following $ C^{2,\alpha} $ estimate of $ Y $ was proved.

\begin{lma} \label{lma-ap-3}
Let  $ \sigma^0 \in \M^{5, \alpha}_+$.
There exists positive numbers $ \delta$, $\epsilon $
and $ C$, depending only on $ \sigma^0$,
with the following properties:

Suppose $ \sigma \in \M^{4, \alpha}_+$ satisfying
$$ || \sigma^0 - \sigma ||_{C^{2, \alpha}} < \delta .$$
Let $ X (\sigma) $ be an isometric embedding of $ (\Sigma, \sigma)$.
Then for any  $ \tau \in \M^{2, \alpha}_+$ satisfying
$$ || \sigma - \tau ||_{C^{2, \alpha}} < \epsilon , $$
the solution  $Y=Y(\sigma,\tau)$ to \eqref{eq-isoembedding}
satisfies
$$ || Y  ||_{C^{2, \alpha} } \leq C || \sigma - \tau ||_{C^{2, \alpha}} .$$

\end{lma}

For the purpose in this paper, we want to obtain the corresponding $ C^{k, \alpha} $ estimate ($k \ge 4$) of $ \phi $ and $ Y$. We have

\begin{lma} \label{lma-ap-5}
Let $ k \ge 4 $ be an integer. Let  $ \sigma_0 \in
\M_+^{k+1, \alpha} $. There exists positive numbers $ \delta$,
$\epsilon $ and $ C$, depending only on $ \sigma_0$, with
the following properties:

Suppose  $ \sigma \in \M^{k, \alpha}_+\cap B(\sigma_0,1)$ where $B(\sigma_0,1)$
 is the open ball in $\M_+^{k, \alpha}$ with center at
 $\sigma_0$ and radius 1.
 Let $X(\sigma)$ be an isometric embedding of $(\Sigma,\sigma)$ in $\R^3$.
 Suppose
$$ || \sigma_0 - \sigma ||_{C^{k, \alpha}} < \delta . $$
 Then for any
 $ \tau \in \M_+^{k, \alpha}\cap B(\sigma_0,1)$ satisfying
$$ || \sigma - \tau ||_{C^{k, \alpha}} <
\epsilon, $$
the solution  $Y=Y(\sigma,\tau)$ to \eqref{eq-isoembedding}
satisfies
$$ ||   Y ||_{C^{k, \alpha} }
\leq C || \sigma - \tau ||_{C^{k, \alpha}}  . $$
Thus, if $ X (\tau ) = X(\sigma)+Y (\sigma, \tau) $ is the corresponding
isometric embedding of $(\Sigma,\tau)$, then
$$ || X( \s) - X ( \tau)  ||_{C^{k, \alpha} }
\leq C || \sigma - \tau ||_{C^{k, \alpha}}  . $$
\end{lma}

\begin{proof} Let $X_0$ be  a fixed isometric embedding of $(\Sigma,\sigma_0)$ so that the origin is the center of the largest inscribed sphere of $X_0 (\Sigma)$ in $ \R^3$. Let $ \{(u, v)\}$ be a fixed coordinates chart of $ \Sigma $
and let $ \Omega \subset \Sigma $ be an open set
whose closure is covered by $\{(u, v)\}$.
On $ \Omega$, we have $|X_0|\ge C$ and $K_0(\la X_0,(X_0)_1 \wedge (X_0)_2\ra)^2\ge C$. Here and below, $C$ always denote a positive constant depending only on $\sigma_0$, $K_0$ denotes the Gaussian curvature of $X_0$, and $(X_0)_1=(X_0)_u$, etc.

By Lemma \ref{lma-ap-3}, there exist positive constants $\delta$, $\e$ and $C $, depending only on $\sigma_0$, such that for any $\sigma, \tau \in \M_+^{k,\alpha}$ with $||\sigma_0-\sigma||_{C^{k,\alpha}}<\delta$ and $||\tau-\sigma||_{C^{k,\alpha}}< \e$,  there exists an
isometric embedding $ X(\sigma)$ of $(\Sigma, \sigma)$ such that
\be \label{X-2a-e1}
 || X(\sigma) - X _0 ||_{C^{2, \alpha} } \le C || \sigma - \sigma_0 ||_{C^{2, \alpha}}
 \ee
and the solution  $Y(\sigma,\tau)$ to \eqref{eq-isoembedding} (with $ X = X(\sigma)$) satisfies
\be\label{Y-2a-e1}
||Y ( \sigma, \tau) ||_{C^{2,\alpha}} \le C||\tau-\sigma||_{C^{2,\alpha}}.
\ee
For such given $ \sigma $ and $ \tau$,  let $ X ( \tau) = X(\sigma) + Y (\sigma, \tau)$ and let $K(\sigma)$, $K(\tau)$ be the Gaussian curvature of $ X (\sigma)$, $X(\tau)$.
Assuming $ \delta$, $ \epsilon$ are sufficiently small, by \eqref{X-2a-e1} and \eqref{Y-2a-e1} we have
$$| X (\sigma) | \ge C, |X(\tau) |\ge C, $$
$$K(\sigma)(\la X(\sigma),(X(\sigma))_1 \wedge (X(\sigma))_2\ra)^2\ge C, $$
$$K(\tau)(\la X(\tau),(X(\tau))_1 \wedge (X(\tau))_2\ra)^2\ge C .$$
Here and below
we always consider points in $ \Omega $.

Consider  $\rho=\frac12|X (\sigma) |^2$ as in \cite[Section 3]{Nirenberg}.
Let
$$
\bA = \rho_{11}-\Gamma^1_{11}\rho_1-\Gamma_{11}^2\rho_2-E
$$
$$
\bB = \rho_{22}-\Gamma_{22}^1\rho_1-\Gamma_{22}^2\rho_2-G
$$
$$
\bC = \rho_{12}-\Gamma_{12}^1\rho_1-\Gamma_{12}^2\rho_2-F,
$$
where $\Gamma_{ij}^k$, $i,j,k \in \{1, 2\}$,  are Christoffel symbols. By the equation (3.7) in \cite{Nirenberg},
\be \label{eq-ellipticity}
\begin{split}
 \bA\bB-\bC^2 &
 = \Delta^2 K(\la X(\sigma),X_3(\sigma)\ra)^2\\
 & = K(\la X(\sigma),X_1(\sigma)\wedge X_2(\sigma)\ra)^2 \\
& \ge  C .
\end{split}
\ee
Differentiate this equation with respect to the $i$-th variable, we have
\be \label{eq-dpde}
\bB\rho_{i11}+\bA\rho_{i22}-2\bC\rho_{i11}=\bP,
\ee
where
$\bP=\bP(\s,\p\s,\p\p\s,\p\p\p\s,\p\rho,\p\p \rho,X(\sigma) ,\p X (\sigma))$
is some fixed polynomial function of its arguments.
Here we used  a basic fact that the $m$-th derivatives of $ X$, $ m \ge 2$, can be
expressed as a linear combination of $X$, $X_1$, $X_2$ with
coefficients involving derivatives of $ \sigma$, $ \rho $ of order at most $m$
(see p.348 in \cite{Nirenberg}).
Now, since $|| \s ||_{C^{3, \alpha}}$, $ || \rho ||_{C^{2,\alpha} }$, $ || X (\sigma) ||_{C^{2, \alpha} }$
are all bounded, it follows from \eqref{eq-ellipticity} and \eqref{eq-dpde}  that
 $ || \rho ||_{ C^{3,\alpha} }$ is bounded. This in turn implies that
 $ || X(\sigma )||_{ C^{3,\alpha} }$ is bounded. Next, since the $ || \s ||_{ C^{4 , \alpha} }$, $|| \rho ||_{ C^{3,\alpha} }$, $ || X (\sigma) ||_{ C^{3, \alpha} }$ are bounded, we see
 $ || \rho ||_{ C^{4,\alpha} } $ is  bounded, which then implies  $ || X (\sigma ) ||_{ C^{4, \alpha} }$ is bounded. Hence,
\begin{equation*}
||\rho||_{C^{4,\alpha} }+||X(\sigma)||_{C^{4,\alpha}}\le C.
\end{equation*}
Continue in this way and use the fact that $ || \s ||_{ C^{k, \alpha} } $ is bounded, we conclude that
\be\label{k-1norm-e1}
||\rho||_{C^{k,\alpha}}+||X(\sigma)||_{C^{k,\alpha}}\le C.
\ee
Similarly, we have
\be\label{k-1norm-e2}
|| \tilde{\rho}||_{C^{k,\alpha}}+||X(\tau)||_{C^{k,\alpha}}\le C,
\ee
where  $\tilde \rho=\frac12|X(\tau)|^2$, and
$ \tilde \rho $ satisfies
\be \label{eq-dpdet}
\tilde \bB \tilde \rho_{i11}+ \tilde \bA \tilde \rho_{i22}-2 \tilde \bC \tilde \rho_{i11}
= \tilde \bP
\ee
where $ \tilde \bA$, $ \tilde \bB$, $ \tilde \bC $, $ \tilde \bP$ are
constructed in the same way as $ \bA$, $ \bB$, $ \bC$,  $ \bP$.
By \eqref{eq-dpde} and \eqref{eq-dpdet}, we have
\be \label{eq-dpdedf}
\begin{split}
 \ &   \ \tilde \bB\lf(\tilde \rho_{i11}-\rho_{i11}\ri)+\tilde \bA\lf(\tilde\rho_{i22}-\rho_{i22}\ri)-2\tilde \bC\lf(\tilde\rho_{i12}-\rho_{i12}\ri) \\
 = & \ \tilde \bP-\bP+(\bB-\tilde \bB)\rho_{i11}+  (\bA-\tilde \bA)\rho_{i22}-2(\bC-\tilde \bC)\rho_{i12}.\\
  \end{split}
\ee
By \eqref{Y-2a-e1}, \eqref{k-1norm-e1} and \eqref{k-1norm-e2},
 we have
\be \label{eq-start}
||\rho-\tilde\rho||_{C^{2,\alpha}}+||X(\sigma)-X(\tau)||_{C^{2,\alpha}}\le C||\s-\tau||_{C^{k,\alpha}}.
\ee
Now suppose for some integer $ l $ satisfying $2\le l<k$, we have
 $$
||\rho-\tilde\rho||_{C^{l,\alpha}}\le C||\s-\tau||_{C^{k,\alpha}}.
$$
By \eqref{k-1norm-e1}, \eqref{k-1norm-e2} and \eqref{eq-start}, we then have
$$
||X(\sigma)-X(\tau)||_{C^{l,\alpha}}\le C||\s-\tau||_{C^{k,\alpha}} ,
$$
where we also used the previously mentioned fact regarding writing the derivatives of $ X(\s) , X( \tau)$
in terms of those of $ \rho, \tilde \rho $ (p.348 in \cite{Nirenberg}).
On the other hand, we have
$$
||\bA-\tilde \bA||_{C^{l-2,\alpha}} \le C\lf(||\rho-\tilde\rho||_{C^{l,\alpha}}+||\s-\tau||_{C^{l-1,\alpha}}\ri),
$$
$$
||\bB-\tilde \bB||_{C^{l-2,\alpha}} \le C \lf(||\rho-\tilde\rho||_{C^{l,\alpha}}+||\s-\tau||_{C^{l-1,\alpha}}\ri),
$$
$$
||\bC-\tilde \bC||_{C^{l-2,\alpha}} \le C\lf(||\rho-\tilde\rho||_{C^{l,\alpha}}+||\s-\tau||_{C^{l-1,\alpha}}\ri),
$$
and
$$
||\bP-\tilde \bP||_{C^{l-2,\alpha}}\le C\lf(||\rho-\tilde\rho||_{C^{l,\alpha}}+||\s-\tau||_{C^{l+1,\alpha}}
+||X(\s)-X(\tau)||_{C^{l-1,\alpha}}\ri).
$$
Since $l+1\le k$, by \eqref{eq-dpdedf}  we conclude
$$
||\rho-\tilde\rho||_{C^{l+1,\alpha}}\le C||\s-\tau||_{C^{k,\alpha}},
$$
and therefore
$$
||X(\s)-X(\tau)||_{C^{l+1,\alpha}}\le C||\s-\tau||_{C^{k,\alpha}}
$$
The result follows by induction.
 \end{proof}

 \begin{lma}\label{phi-l0}
 Let $\sigma_0\in \M_+^{k+1,\alpha}$ ($k \ge 4)$.
 Let $ \epsilon > 0 $ be as in Lemma \ref{lma-ap-5}.
 Suppose $\sigma\in \M_+^{k,\alpha}$ and  $||\sigma-\sigma_0||_{C^{k,\alpha} }<\e$.
Let $ X = X( \sigma) $ be an isometric embedding of $(\Sigma, \sigma)$ into $ \R^3$.
Given any  $\rho\in\Sp^{r,\alpha} $ ($r\ge 2$),
let  $ \phi = \Psi (\sigma, \rho) $ be  the unique solution of \eqref{phi-e1} which is $L^2$-orthogonal to the coordinates functions of the unit normal of $X(\sigma)$;
let $Y = \Phi(\sigma,\rho)$  be the unique solution of \eqref{eq-truelinear} which vanishes at a fixed point on $ \Sigma $ and  is obtained by integrating $ Y_u $ and $Y_v $ defined by  \eqref{Y-e1} and \eqref{p-e1}.
 There exist   $C>0$ depending only  on $\sigma_0$ and $ \ep$
 such that
   \be \label{eq-rs}
||\phi||_{C^{s,\alpha}}\le C  ||\rho||_{C^{r,\alpha}}
  \ee
and
\be  \label{eq-estofPhi}
||\Phi(\sigma,\rho)||_{C^{s,\alpha} } \le C||\rho||_{ C^{r,\alpha} }
\ee
where $s=\min\{r,k-1\}$.
\end{lma}

\begin{proof}
Let $ X_0 $ be a fixed isometric embedding of $ (\Sigma, \sigma_0)$.
By Lemma \ref{lma-ap-5},  we may assume $ X = X(\sigma) $
is chosen such that
 $||X - X_0||_{C^{k,\alpha} } \le C$, where $ C $  depends only on $\sigma_0$
 and $ \epsilon$.
 Recall that $\phi$ satisfies
\be\label{phi-e2}
\frac{1}{\Delta}\lf(\frac
  N{K\Delta}\phi_u-\frac{M}{K\Delta}\phi_v\ri)_u-
\frac{1}{\Delta}\lf(\frac
 M{K\Delta}\phi_u-\frac{L}{K\Delta}\phi_v\ri)_v+H\phi=
 \mathcal{L}(c_1,c_2) -T .
\ee
By \eqref{L-e1}, \eqref{c-e1} and \eqref{T-e1}, we have
\be
 ||  \mathcal{L}(c_1,c_2) -T ||_{ C^{s-2, \alpha} }\le C ||
\rho ||_{C^{r, \alpha} } .
\ee
Hence,
\be \label{eq-rs2}
||\phi||_{C^{s,\alpha }}\le C \lf(||\phi||_{C^0}+||\rho||_{C^{r,\alpha} } \ri).
\ee
Therefore, \eqref{eq-rs} holds by \eqref{eq-rs2} and Lemma \ref{phi-0-l1}.
Now  \eqref{eq-estofPhi} follows directly from
 \eqref{Y-e1}, \eqref{p-e1} and \eqref{eq-rs}.
\end{proof}

Now we are in a position to prove the main result of this section.

 \begin{thm}\label{C1-t1}
 Let $\sigma_0\in \M_+^{k+1,\alpha}$ ($k \ge 4 $). There exists a constant $\kappa>0$ such that the map
 $$\mathcal{F}:\M_+^{k, \alpha}\to   \Sp^{k-3,\alpha} $$
 defined by \eqref{dofmapF}  is $ C^1 $  in
 $U=\{\sigma\in \M_+^{k,\alpha} \ | \ ||\sigma-\sigma_0||_{k,\alpha}<\kappa\}$.
   \end{thm}

   \begin{proof}
   Let $ \e >0 $ and $ \delta > 0$ be as  in Lemma \ref{lma-ap-5}.
   We may assume that $ \e $  is so small that the open set
   $ U_{2 \epsilon} = \{ \hat{\s} \in \Sp^{k, \alpha} \ | \ || \hat{\s} - \sigma_0 ||_{C^{k,\alpha}} < 2 \e \} $
   in $ \Sp^{k, \alpha} $ is indeed contained in $ \M^{k, \alpha}_+$.

Let $ \kappa > 0 $ be chosen such that $ \kappa < \min\{ \epsilon, \delta \}$.
Suppose $\sigma\in U $.  Let $ X = X ( \sigma ) $ be an isometric embedding of $ (\Sigma, \sigma)$.
Since $ \kappa < \epsilon$,
we may assume that $ X (\sigma) $ is chosen such that
$ || X (\sigma ) ||_{C^{k, \alpha} } \le C $,
where $ C $ depends only on $ \sigma_0 $ and $ \epsilon$.
Given any $\eta\in \mS^{k,\alpha}$ such that $||\eta||_{C^{k,\alpha} }=1$,
consider $ \sigma + t \eta \in \M^{k, \alpha}_+$ for $ | t |< \epsilon$.
 Let
 $ X(\sigma+t\eta)=X(\sigma)+ Y (\sigma, \sigma + t\eta) $
 be the (nearby) isometric embeddings of  $(\Sigma, \sigma+t\eta )$.
 In what follows,  we  write $ P = Y ( \sigma, \sigma + t \eta )$.
 By \eqref{eq-isoembedding}, $P $ satisfies
 \be \label{eq-pdeofp}
 2 d X \cdot d P = t \eta - ( d P )^2.
 \ee
Since $ \kappa < \delta $ and $ | t | < \epsilon $,
by Lemma \ref{lma-ap-5} we have
    \be \label{P-e1}
    || P  ||_{C^{k,\alpha} }=||X(\sigma+t\eta)-X(\sigma)||_{C^{k,\alpha }} \le
    C  ||t\eta||_{C^{k,\alpha }} \le C|t|
    \ee
     where $C >0$ is the constant  in Lemma \ref{lma-ap-5}. In particular, $ C $
     is independent on $ \eta$.


Now let $ Y = \Phi (\sigma, \eta)$ be the solution to
    \be\label{linearlized-e1}
    2dX \cdot dY=\eta .
  \ee
By \eqref{eq-pdeofp} and \eqref{linearlized-e1}, we have
    \be
    2dX\cdot (dP-tdY)=  -( dP )^2 :=\rho .
    \ee
 Since $ P $ is of $ C^{k, \alpha}$, we know
$ \rho \in \mS^{k-1, \alpha}$. By \eqref{P-e1},
\be \label{rho-e1}
 || \rho ||_{ C^{k-1, \alpha} } \le C t^2 .
 \ee
We claim that $ P - t Y  = \Phi (\sigma, \rho)$.
    To see this, we first recall that
    $ P = Y ( \sigma, \sigma + t \eta) = \lim_{m \rightarrow \infty} Y_m  $
    in the $ C^{2, \alpha}$ norm,   where $ Y_0 = 0 $,
    $ Y_m = \Phi( \sigma, \rho_{m-1} ) $ and
    $ \rho_{m-1} = t \eta - ( d Y_{m-1} )^2  $.
Next, let $ \phi_m $ be the corresponding unique solution $ \phi $
of \eqref{phi-e1} with $ \rho $ replaced by $ \rho_{m-1} $.
By \eqref{phiandPhi}, $ \phi_m $ satisfies
\be
\phi_m ( u, v ) = \frac{1}{\Delta} \lf[ X_v \cdot (Y_m)_u - X_u \cdot (Y_m)_v \ri].
\ee
Let $ \phi_P$ be given by
\be
\phi_P ( u, v ) = \frac{1}{\Delta} \lf( X_v \cdot P_u - X_u \cdot P_v \ri) .
\ee
Since $ Y_m $ converges to $ P $ in the $C^{2, \alpha}$ norm,
we see that $ \phi_m $ converges to $ \phi_P $ in the $ C^{1, \alpha} $ norm.
In particular, $ \phi_P $ is $L^2$-orthogonal to the coordinate functions of $X_3$.
On the other hand,  by (6.15) in \cite{Nirenberg}, $ \phi_P $ is a solution to
\eqref{phi-e1} with $ \rho $ replaced by $ \tilde{\rho} = t \eta - ( d P )^2 $.
Hence, by definition, we have
$ \phi_P = \Psi( \sigma, \tilde{\rho} )$. Since $ P $ also vanishes at the fixed point
where $ Y_m $ is set to vanish, we know that $ P $ is obtained by integrating $ P_u $ and $ P_v $, which are given by  \eqref{Y-e1} and \eqref{p-e1} with $ \rho $ replaced by
$ \tilde{\rho} $ and with $ \phi $ replaced by $ \phi_P = \Psi( \sigma, \tilde{\rho} )$.
By definition, this shows $ P = \Phi( \sigma, \tilde{\rho} )$. Therefore, we have
\be
P - t Y = \Phi( \sigma, \tilde{\rho} ) -  t \Phi( \sigma, \eta) =
\Phi ( \sigma, \rho) .
\ee
By Lemma \ref{phi-l0} and \eqref{rho-e1}, we then have
    \be
    ||P-tY||_{C^{k-1,\alpha}}\le Ct^2
    \ee
    or equivalently
   \be \label{xzest}
   || X (\sigma + t \eta) - Z (t) ||_{C^{k-1,\alpha}}\le Ct^2
   \ee
   where $ Z ( t ) = X (\sigma) + t Y $.

   Next,  applying the fact that the second fundamental form $ \Pi ( Z )$ of any
   $ Z \in \E^{m, \alpha} $ ($m \ge2$), written in local coordinates,
   are polynomial functions of derivatives of $ Z $ of order at most $2$,
we  see from \eqref{xzest} and the fact
   $ || X (\sigma + t \eta) ||_{C^{k,\alpha}} \le C$    that
  \be \label{xztest-2}
   || \Pi ( X (\sigma + t \eta) ) - \Pi ( Z (t) ) ||_{C^{k-3, \alpha} } \le C t^2 .
   \ee
   On the other hand, because the map $ \rho \mapsto \Phi (\sigma, \rho) $ is
    linear from $ \Sp^{k, \alpha} $ to $ \X^{k-1, \alpha}$,
   and because $ || \Phi (\sigma, \eta) ||_{C^{k-1, \alpha}} \le C $,
    there is  a linear map $A:  \Sp^{k,\alpha}
    \to \Sp^{k-3,\alpha}$ such that
    \be \label{xztest-3}
   || \Pi(Z(t))-\Pi(X(\s) ) - t A(  \eta)||_{C^{k-3,\alpha}}\le C t^2.
   \ee
 By \eqref{xztest-2} and \eqref{xztest-3}, we have
  \be
   ||\mathbb{ II}(X(\sigma+t\eta))-\mathbb{ II}(X(\sigma))- t A(\eta)||_{C^{k-3,\alpha}}\le Ct^2
    \ee
    for all $\eta\in \Sp^{k,\alpha}$ with $ || \eta ||_{C^{k, \alpha} } = 1 $.

We want to compute $A(\eta)$ explicitly, which is simply $ \frac{d}{dt}|_{t=0} \Pi ( Z (t) ) $. Since $ A (\eta) $  also depends on $\sigma$, we will denote it by $A^{(\s)}(\eta)$. Let $e_3(t)=\frac{Z_1(t)\wedge Z_2(t)}{|Z_1(t)\wedge Z_2(t)|}$ be the unit normal of
$Z(t)$, where $Z_1=\frac{\p Z}{\p u}$ and $Z_2=\frac{\p Z}{\p v}$. Let $ i, j \in \{ 1, 2\}$ and let $ Z_{ij} $ denote the corresponding second order derivative of $ Z$. Then
   $$
   \mathbb{ II}(Z(t))_{ij}=-\la e_3, Z_{ij}\ra.
   $$
   Hence
   $$
   A(\eta)_{ij}=-\la e_3(0), Y_{ij}\ra-\la \frac{d e_3}{dt}, Z_{ij}\ra|_{t=0}.
   $$
Since $ \frac{d e_3}{dt}  \perp e_3 $, we may assume
$ \frac{d e_3}{dt} |_{ t = 0 } = c^i  X_i  $
for some coefficients $ c^i$.
Then
\bee
- \la e_3, \frac{d Z_j}{dt}  |_{ t = 0 } \ra =
\la \frac{d e_3}{dt}  |_{ t = 0 } ,  X_j \ra = c^i \sigma_{ij}.
\eee
 Thus,
\bee
c^i =  -   \sigma^{ij} \la e_3, Y_j \ra.
\eee
Therefore
 \be\label{C1-e1}
 A^{(\s)}(\eta)_{ij}=-\la X_3(\sigma), Y_{ij}\ra+\sigma^{kl} \la X_3(\s), Y_k \ra
 \la X_l, X_{ij} \ra
 \ee
 where $X_3(\s)$ is the unit normal of $X(\s)$.

Using the facts that $||Y||_{k-1,\alpha}\le C$ (Lemma \ref{phi-l0}) and $||X(\s)||_{k,\alpha}\le C$, where both constants $C$ depend only on $\sigma_0$,  we conclude from \eqref{C1-e1} that $A^{(\s)}$ is a bounded linear map from $\Sp^{k,\alpha}$ to $\Sp^{k-3,\alpha}$.

Next we want to prove that the map $ \sigma \mapsto A^{(\s)}$ is continuous  in the operator topology. Namely, for $\s_1\in U$,  we want to prove that
 \be\label{C1-e2}
 \lim_{\sigma\in U,\sigma\to\sigma_1}\sup_{\eta\in \Sp^{k,\alpha},||\eta||_{C^{k,\alpha}}=1} ||A^{(\s)}(\eta)-A^{(\s_1)}(\eta)||_{C^{k-3,\alpha}}=0.
 \ee
We first note that $A^{(\s)}$ does not depend on any particular choice of the embedding $ X(\sigma) $.
Suppose $\s_1\in U$ and suppose $ X(\sigma_1 )$ is a fixed isometric embedding
of $\sigma_1$ such that $ || X (\sigma_1) ||_{C^{k, \alpha}} \le C$.
 By Lemma \ref{lma-ap-5}, for any $ \sigma \in \M^{k, \alpha}_+$
 with  $||\s-\s_1||_{C^{k,\alpha}} < \e - \kappa$,  an isometric embedding $X(\s)$ can be chosen such that
 $ X(\s) = X(\s_1)+ P_1 $, where $ P_1 = Y (\sigma_1, \sigma) $ and
\be\label{C1-e3}
||P_1||_{C^{k,\alpha}}\le C||\s-\s_1||_{C^{k,\alpha}} .
 \ee
Here and below all the constants $C$ depend only on  $\s_0$, but not on
$ \sigma $ and $ \eta$.

For any given $\eta \in \Sp^{k,\alpha}$ with $ || \eta ||_{C^{k, \alpha} } =1$,
let $Y^{(1)} = \Phi (\sigma_1, \eta)$ and $Y = \Phi (\sigma, \eta) $ be the solutions of
$$
2dX(\s_1)\cdot dY^{(1)}=\eta
$$
and
$$
2dX (\s)\cdot dY=\eta.
$$
In order to prove \eqref{C1-e2}, by \eqref{C1-e1} and \eqref{C1-e3}, it is  sufficient to prove that
\be\label{C1-e4}
||Y^{(1)}-Y||_{C^{k-1,\alpha}}\le C||\s-\s_1||_{C^{k,\alpha}} .
\ee

Let $\phi^{(1)} = \Psi(\sigma_1, \eta) $ and $\phi = \Psi( \sigma, \eta) $ be the
functions that are used to construct $ Y^{(1)} $ and $ Y $.
Then $\phi^{(1)}$ and $\phi$ satisfy two elliptic PDEs
$$
a_{ij}^{(1)}\phi_{ij}^{(1)}+b_i^{(1)}\phi_i^{(1)}+c^{(1)}\phi^{(1)}=f^{(1)}
$$
and
$$
a_{ij}\phi_{ij} +b_i \phi_i+c \phi =f ,
$$
which correspond to \eqref{phi-e1} (where the metric and the embedding involved are given by $ \sigma_1$ and $ X(\sigma_1)$, $ \sigma$ and $ X(\sigma)$ respectively,
and $ \rho $ is replaced by $ \eta$).
By \eqref{L-e1}-\eqref{T-e1}, \eqref{C1-e3} and the fact $ || X(\sigma_1) ||_{C^{k, \alpha}} \le C$,  we have
\be \label{C1-e50}
||a_{ij}^{(1)} ||_{C^{k-2,\alpha}}
+||b^{(1)}_i ||_{C^{k-3,\alpha}}
+||c^{(1)} ||_{C^{k-2,\alpha}}  \le  C
\ee
and
\be \label{C1-e5}
\begin{split}
\ & ||a_{ij}^{(1)}-a_{ij}||_{C^{k-2,\alpha}}
+||b^{(1)}_i-b_i||_{C^{k-3,\alpha}}
+||c^{(1)}-c||_{C^{k-2,\alpha}}+||f^{(1)}-f||_{C^{k-3,\alpha}}\\
& \le  C ||\s_1-\s||_{C^{k,\alpha}} .
\end{split}
\ee
Hence
\be \label{constpde}
a_{ij}^{(1)} (\phi_{ij}^{(1)}-\phi_{ij}  )+b_i^{(1)}(\phi_i^{(1)}-\phi_i  )+c^{(1)}(  \phi^{(1)}-\phi)=q
\ee
where $q=f^{(1)}-f +(  a_{ij}-a_{ij}^{(1)}) \phi_{ij}+(  b_i-b_i^{(1)}) \phi_i+(  c-c^{(1)})\phi$. By  \eqref{C1-e5} and Lemma \ref{phi-l0}, we have
\be\label{C1-e6}
||q||_{k-3,\alpha}\le C||\s_1-\s||_{C^{k,\alpha}}.
\ee
It follows from \eqref{C1-e50}, \eqref{constpde}, \eqref{C1-e6} and the Schauder estimates that
\be\label{C1-e7}
||\phi^{(1)}-\phi||_{C^{k-1,\alpha}}\le C\lf(||\phi^{(1)}-\phi||_{C^0}+||\s_1-\s||_{C^{k,\alpha}}  \ri).
\ee
To estimate $ || \phi^{(1)}-\phi||_{C^0} $, let $x_1, x_2, x_3$ be coordinate functions of the unit normal of $X(\s_1)$ and let $y_1$, $y_2$, $y_3$ be the unit normal of $X(\s)$.  Define
$$\beta_i=\int_{\Sigma}x_i (\phi^{(1)}-\phi)d\s_1, \ \
\omega_{ij}=\int_{\Sigma}x_ix_j d\s_1.
$$
 Since
$$
\int_{\Sigma}x_i  \phi^{(1)} d\s_1=\int_{\Sigma}y_i  \phi d\s=0 ,
$$
we have
\be \label{lowerbeta}
|\beta_i|\le C ||\s_1-\s||_{C^{k,\alpha}}
\ee
where we have also used  \eqref{C1-e3} and Lemma \ref{phi-0-l1}.
Since $(\omega_{ij})$ has an inverse $(\omega^{ij})$, we
let $\beta^i=\omega^{ij}\beta_j$. Then
$$\phi^{(1)}-\phi-\sum_{k}\beta^k x_k$$
is $L^2$-orthogonal to each $x_i$.
Moreover,
\be \label{upbeta}
|\beta^i|\le C||\s_1-\s||_{C^{k,\alpha}}
\ee
and
$$
a_{ij}^{(1)} (\phi^{(1)}-\phi -\sum_{k}\beta^kx_k  )_{ij}+b_i^{(1)}(\phi^{(1)}-\phi -\sum_{k}\beta^kx_k  )_i+c^{(1)}(  \phi^{(1)}-\phi-\sum_{k}\beta^kx_k)=q_1
$$
where  $ q_1 $ is some function satisfying
$||q_1||_{C^{k-3,\alpha}}\le C||\s_1-\s||_{C^{k,\alpha}}$.
By the integral expression of $\phi_{ij}^{(1)}-\phi_{ij}-\sum_{k}\beta^kx_k$ in terms of the Green's function, see \cite{Nirenberg}, we have
\be
||\phi^{(1)}-\phi -\sum_{k}\beta^kx_k||_{C^0}\le C ||\s_1-\s||_{C^{k,\alpha}},
\ee
 and therefore
 \be \label{czeroest}
 ||\phi^{(1)}-\phi||_{C^0}\le C  ||\s_1-\s||_{C^{k,\alpha}}
 \ee
 by \eqref{upbeta}.
 It follows from  \eqref{C1-e7} and \eqref{czeroest} that
 \be\label{C1-e8}
||\phi^{(1)}-\phi||_{C^{k-1,\alpha}}\le C ||\s_1-\s||_{C^{k,\alpha}}.
\ee
Finally, because $ Y^{(1)} $ and $ Y $ are obtained by integrating
$ (Y^{(1)})_u,  (Y^{(1)})_v$ and  $Y_u , Y_v$
which are determined by \eqref{Y-e1} with the corresponding
$ \phi^{(1)}$ and $ \phi $ inserted,
we conclude from \eqref{C1-e8}
that \eqref{C1-e4} is true, hence the map $ \s \mapsto A^{(\s)}$ is continuous
in the operator topology.    \end{proof}

\section{Existence of critical points on  nearby surfaces} \label{smallsolution}

We are now in a position to apply Theorem \ref{C1-t1} and
the {IFT} to study {\bf Question 2}.
 Let $\Sigma$, $ N $ be given as in the Introduction, namely,
$\Sigma$ is a smoothly embedded, closed, spacelike two-surface, which is topologically a two-sphere,
in a smooth time-oriented spacetime $N$. Suppose the mean curvature vector $ H $ of $ \Sigma $ in $ N $
is spacelike. Let $ \sigma $ be the induced metric on $ \Sigma $ from $ N $.
Suppose $ \tau_0 $ is a $ C^{k+1, \alpha} $ function on $ \Sigma $ with $ k \ge 5  $
such that  $ \sigma + d \tau_0 \otimes d \tau_0 $ has positive Gaussian curvature and $ \tau_0 $
is a solution to  \eqref{1stvariation-e1} on $ \Sigma $.

To describe spacelike two-surfaces which are ``close" to $ \Sigma$, we use the exponential map
$ exp^N( \cdot )$ associated to the Levi-Civita connection of the Lorentzian metric $ g $
on $ N $. Precisely, we first fix a smooth
future timelike normal vector field $ J$  on $ \Sigma $ which is orthogonal to $ H $.
Then $ \{ H, J \}$ form a basis for the normal bundle $ (T \Sigma)^\perp $ of $ \Sigma $.
Let $\mathcal{B}=C^{k,\alpha}(\Sigma)\times C^{k,\alpha}(\Sigma)$,
where $C^{k,\alpha} (\Sigma ) $ is the Banach space of $C^{k,\alpha}$ functions on $\Sigma$. For any constant $ a > 0 $,
let $B(a)$ be the open ball in $\mathcal{B}$ centered at $(0,0)$ with radius $a$.
If $ a $ is sufficiently small, for any $ f = ( f_1, f_2) \in B (a) $,  the map
$ F_f:  \Sigma \rightarrow N $
defined by $ F_f (x) = \exp^N (f_1(x) H (x)+f_2(x) J (x))$ is
a $ C^{k, \alpha} $ embedding,  moreover $ F_f ( \Sigma ) $ remains to be spacelike and
has spacelike mean curvature vector $ H_f$.

Consider the  map
$ \mathcal{I}: B(a)\to \mathcal{M}^{k-1,\alpha} (\Sigma) $
given by $\mathcal{I}(f)=F_f^*(g)$, where $ \mathcal{M}^{k-1, \alpha} (\Sigma) $ denotes the space of $ C^{k -1, \alpha} $ Riemannian metrics on $ \Sigma$.
Let $ U_{\tau_0} ( a) $ be the open ball in $ C^{k, \alpha} (\Sigma) $ centered at $ \tau_0 $ with
radius $ a $. For $ a $ sufficiently small, we may also assume that
$ \mathcal{I}(f)  + d \tau \otimes d \tau $
is a metric of positive Gaussian curvature for all
 $ f \in B(a) $ and $ \tau \in U_{\tau_0} (a) $.
Given such a small $ a $, we define the map
$$
\mH :  B(a)\times U_{\tau_0}(a) \longrightarrow C^{k-4, \alpha} ( \Sigma)
$$
where
\bee \label{mH-def}
\begin{split}
 \ & \mH( f, \tau )  \\
 =  & - \lf[ \hat{H}\hat{\sigma}^{ab} - \hat{\sigma}^{ac} \hat{\sigma}^{bd} (\hat{h}_{cd})  \ri]
\frac{ \nabla_b \nabla_a \tau }{ \sqrt{ 1 + | \nabla \tau |^2 } }
+ \div_{\Sigma } \lf[ \frac{ \nabla  \tau}{\sqrt{ 1 + | \nabla  \tau |^2} } \cosh \theta |   H_f | - \nabla  \theta - V_f \ri]
\end{split}
\eee
which is just the left side of \eqref{1stvariation-e1} but with $ \sigma $ replaced by $ ^f\sigma =  \mathcal{I}(f) $,
$ H $ replaced by $ H_f $
and $ V $ replaced by $ V_f$.
Here the vector field $ V_f $ on $ \Sigma $ is understood as the pull back,  through the embedding $ F_f $, of the vector field  dual to
the one form $ \alpha^N_{e^{H}_3} ( \cdot) $ on $ F_f (\Sigma) $.

\begin{prop}\label{mH-smooth-p1} $\mH$ is a $C^1$ map.
\end{prop}
\begin{proof}
Note that $\mathcal{I}$ is a $C^1$ map. Hence, the map
$ (f, \tau)  \mapsto \hat{\sigma} = ^f\sigma + d \tau \otimes d \tau $
is $ C^1 $ from $ B(a)\times U_{\tau_0}(a) $ to $ \mathcal{M}^{k-1,\alpha} (\Sigma) $.
By Theorem \ref{C1-t1},   the map
$(f, \tau) \mapsto (\hat{h}_{cd}) $
is $ C^1 $ from $ B(a)\times U_{\tau_0}(a) $ to the space of $ C^{k-4, \alpha} $ symmetric $(0,2)$ tensors
on $ \Sigma $. Thus, to show $ \mH $ is $ C^1 $, it only remains to check that the map
$f \mapsto \div_{\Sigma} V_f$ is $C^1$ from  $B(a)\times U_{\tau_0}(a)$ to $C^{k-4,\alpha} (\Sigma) $.

Let $T$ be a smooth future timelike unit vector field on $N$.
Let $ \{ (x^1, x^2) \}$ be any local coordinates on $ \Sigma $. Let $ v_b = (F_f)_*( \frac{\p }{\p x_a } ) $, $ b = 1, 2 $. Then
$$ H_f = (^f\sigma)^{ab} \nabla^N_{v_a} v_b  -
( ^f\sigma )^{cd} \la (^f\sigma)^{ab} \nabla^N_{v_a} v_b, v_c \ra  v_d ,
$$
$$ V_f = (^f \sigma)^{ab} \la \nabla^N_{v_a} e^{H_f}_3, e^{H_f}_4 \ra v_b ,$$
where $ e^{H_f}_3 = - H_f / | H_f | $ and
$ e^{H_f}_4 =   w /  \sqrt{-\la w,w\ra} $
with
$$  w= T -  ( ^f\sigma)^{ab} \la T, v_a \ra v_b  -  \la T,e_3^{H_f} \ra e_3^{H_f} . $$
From this it is easily seen that $f \mapsto \div_{\Sigma} V_f$ is a $C^1$ map.
\end{proof}

\begin{lma}\label{mH-l1}
Let $ d v_{^f\sigma} $, $ d v_{\sigma} $ be the volume form of  $ ^f\sigma $, $ \sigma $ on $ \Sigma $.
Then,  for any $(f,\tau)\in B(a)\times U_{\tau_0}(a)$,
$$
\int_{\Sigma} \mH(f,\tau) dv_{^f\sigma}=0 .
$$
\end{lma}
\begin{proof}
It suffices to verify
\be \label{divergezero}
 \int_{\Sigma}\lf[ \hat{H}\hat{\sigma}^{ab} - \hat{\sigma}^{ac} \hat{\sigma}^{bd} (\hat{h}_{cd})  \ri]\frac{ \nabla_b \nabla_a \tau }{ \sqrt{ 1 + | \nabla \tau |^2 } }dv_{ ^f\sigma} = 0 .
 \ee
Let $ d v_{\hat{\sigma}} $ be the volume of  $\hat \sigma= ^f\sigma+d\tau\otimes d\tau$. Then
$$
\sqrt{1+|\nabla \tau|^2}dv_{\hs}= dv_{ ^f\sigma}
$$
and
$$
\hat\nabla_b\hat\nabla_a \tau=\frac{1}{1+|\nabla \tau|^2}\nabla_b\nabla_a\tau,
$$
where $ \nabla$, $ \hat{\nabla} $ denote covariant derivatives of $ ^f\sigma$, $ \hat{\sigma}$
respectively.
Hence
\bee
\lf[ \hat{H}\hat{\sigma}^{ab} - \hat{\sigma}^{ac} \hat{\sigma}^{bd} (\hat{h}_{cd})  \ri]\frac{ \nabla_b \nabla_a \tau }{ \sqrt{ 1 + | \nabla \tau |^2 } }dv_{ ^f \sigma}=\lf[ \hat{H}\hat{\sigma}^{ab} - \hat{\sigma}^{ac} \hat{\sigma}^{bd} (\hat{h}_{cd})  \ri]\hat\nabla_b \hat\nabla_a \tau  dv_{\hs}
\eee
which implies \eqref{divergezero}
 because $\hat{H}\hat{\sigma}^{ab} - \hat{\sigma}^{ac} \hat{\sigma}^{bd}  \hat{h}_{cd}$ is divergence free with respect to $\hs$.
\end{proof}

In what follows, we assume that $ \tau_0 = 0 $ is a solution to \eqref{1stvariation-e1} on $ \Sigma$.
We give a sufficient condition that guarantees
the existence of solutions to \eqref{1stvariation-e1} on the nearby surfaces $ F_f (\Sigma )$.

\begin{thm}\label{criticalptsexist-t1}
With the above assumptions and notations, suppose  the induced metric $ \sigma $ on $ \Sigma $ has positive Gaussian curvature and the vector field $ V $ on $ \Sigma$ satisfies
$ \div_\Sigma V = 0 $.
Suppose in addition there exists a constant $C>0$ such that
\be\label{assumption-e1}
 \int_\Sigma \lf[ \frac{ ( \Delta \eta )^2 }{ |  H |  } + ( H_0 - |   H | ) | \nabla \eta  |^2 - \Pi_0 ( \nabla \eta ,\nabla \eta)\ri]dv_{\Sigma}\ge C\int_\Sigma (\Delta\eta)^2dv_{\Sigma}
\ee
for all $\eta\in W^{2,2}(\Sigma)$, where $H_0$ and $\Pi_0$ are the mean curvature and the second fundamental form of $(\Sigma,\sigma)$ when isometrically embedded in $\R^3$.
Then for any $ k \ge 5 $ and $ 0 < \alpha < 1$, there exists a small constant $a>0$ such that, for any $f \in B(a)$,  there exists  a $ C^{k, \alpha} $ solution $ \tau $ to \eqref{1stvariation-e1} on the $ C^{k,\alpha}$ embedded surface  $F_f (\Sigma)$.
\end{thm}

\begin{proof}
Since $ \div_\Sigma V  = 0 $, we know $ \tau_0 = 0 $ is a solution to \eqref{1stvariation-e1}
on $ \Sigma$.  For the given $ k $ and $ \alpha$,  let $a>0$ be sufficiently small such that the map $ \mH $ is well defined on $ B(a) \times U_{\tau_0} (a) $ with $ \tau_0 = 0 $. Let
$$D (a) =\{\tau\in C^{k,\alpha} (\Sigma) \ | \  ||\tau||_{C^{k,\alpha}} <a \ \text{\rm and }\int_\Sigma \tau \ dv_{\Sigma}=0\} \subset U_0(a)
$$
and
$$
C^{k-4,\alpha}_0 (\Sigma) =\{\phi\in C^{k-4,\alpha} (\Sigma) \ | \  \int_\Sigma \phi \ dv_{\Sigma}=0\}.
$$
For $(f,\tau)\in B(a)\times D$, define
$$
\mH_0(f,\tau)=\frac{dv_{^f\sigma}}{dv_{\Sigma}}\mH(f,\tau).
$$
By Proposition \ref{mH-smooth-p1} and Lemma \ref{mH-l1}, $\mH_0$ is a $C^1$ map from $B(a)\times D (a) $ to $C^{k-4,\alpha}_0 (\Sigma) $.

Direct computations show that the partial derivative $D_\tau \mH_0 |_{(0,0)}$ of $\mH_0$ at $ (0,0) $ with respect to $ \tau $ is given by
\be\label{partiald-e1}
D_\tau \mH_0|_{(0,0)}(\eta)=-\la H_0\sigma -\mathbb{II}_0,\nabla^2\eta\ra+\div_{\Sigma}\lf(|H|\nabla \eta\ri)+\Delta\lf(\frac{\Delta \eta}{|H|}\ri)
\ee
for $\eta\in  C^{k,\alpha}_0 (\Sigma) $, the space of all $C^{k,\alpha}_0$ functions with zero integral on $(\Sigma, \sigma)$.
Clearly, $D_\tau \mH_0|_{(0,0)}$ is a bounded linear map from $C^{k,\alpha}_0 (\Sigma) $ to $C^{k-4,\alpha}_0 (\Sigma)$. We claim that, under the condition \eqref{assumption-e1},
$D_\tau \mH_0 |_{(0,0)}$ is a bijection. Once this claim can be verified, Theorem \ref{criticalptsexist-t1}  will follow from the implicit function theorem.

To show $D_\tau \mH_0 |_{(0,0)}$ is a bijection, we let $\wt  W^{2,2}(\Sigma)$ be the closed subspace of
$W^{2,2}(\Sigma)$ consisting of those $ \eta $    with $\int_\Sigma \eta =0$.
Here and below, integrations and differentiations are taken with respect to the metric $\sigma $ and we omit witting
the volume form $ d v_\Sigma $ in the integrals.  Consider the following bilinear form on the Hilbert space $\wt  W^{2,2}(\Sigma)$:
\be
\begin{split}
B(\eta,\phi)
= & \ \int_\Sigma D_\tau \mH_0|_{(0,0)} (\eta) \phi \\
= & \
\int_\Sigma     \frac{\Delta \eta \Delta \phi}{ |  H | }
+  (H_0 \sigma - \Pi_0)(\nabla \eta,\nabla\phi)-
  |   H |\la\nabla\eta, \nabla \phi \ra.
 \end{split}
 \ee
  Obvious $B$ is bounded. That is $|B(\eta,\phi)|\le   C_1||\eta||_{W^{2,2}}||\phi||_{W^{2,2}} $ for some constant $C_1$ and for all $\eta, \phi\in  \wt  W^{2,2}(\Sigma)$.

By the $L^p$ estimate \cite[Theorem 9.11]{Gilbarg-Trudinger}, there is a constant $C_2$ such that for all $\eta\in  \wt  W^{2,2}(\Sigma)$
  $$
  ||\eta||_{  W^{2,2}}\le C_2(||\eta||_{L^2}+||\Delta\eta||_{L^2}).
  $$
  Since $\int_\Sigma\eta=0$, by \eqref{gradandlap} we have
  $$
  \int_\Sigma |\Delta\eta|^2\ge \lambda_1^2 \int_\Sigma\eta^2
  $$
  where $\lambda_1>0$ is the first  nonzero  eigenvalue of the Laplacian of $ \sigma$. Hence, by
  \eqref{assumption-e1}, we have
  \be\label{coercive-e1}
  B(\eta,\eta) \ge C_3|| \eta||^2_{W^{2,2}}
 \ee
  for some $C_3>0$ and for all  $\eta\in \wt W^{2,2}(\Sigma)$, i.e. $B$ is coercive. This readily
  implies that $D_\tau \mH_0|_{(0,0)}$ is injective.

  Now let $ f $ be an arbitrary element in $ L^2(\Sigma)$ with $\int_\Sigma f=0$. Define $T: \wt W^{2,2}(\Sigma)\to \R$ by
  $$
  T(\phi)=\int_\Sigma \phi f.
  $$
Since  $T$ is a bounded linear functional on $\wt W^{2,2}(\Sigma)$,
there exists an $\eta\in \wt W^{2,2}(\Sigma)$ such that
  $$
  B(\eta,\phi)=T(\phi)
  $$
  for all $\phi\in \wt W^{2,2}(\Sigma).$ That is to say,
$$
\int_\Sigma     \frac{\Delta \eta \Delta \phi}{ |   H | }
+  (H_0 \sigma - \Pi_0)(\nabla \eta,\nabla\phi)-
  |   H |\la\nabla\eta, \nabla \phi \ra=\int_\Sigma \phi f
  $$
  for all $\phi \in \wt W^{2,2}(\Sigma)$. Integrating by parts, we have

\begin{equation}\label{wsol-e1}
   \int_\Sigma     \frac{\Delta \eta \Delta \phi}{ |  H | }
= \int_\Sigma \lf[\la H_0 \sigma - \Pi_0, \nabla^2 \eta\ra-
  \div_\Sigma(|  H | \nabla\eta)   +f  \ra\ri]\phi
\end{equation}
where we have used the fact that $H_0 \sigma - \Pi_0$ is divergence free with respect to
$ \sigma$. This same fact also implies
$$\int_\Sigma  \la H_0 \sigma - \Pi_0, \nabla^2 \eta\ra-
  \div_\Sigma(|   H | \nabla\eta)   +f =0
  $$
  because $\int_\Sigma f=0$.
  Therefore, if we let
  $$h=\la H_0 \sigma - \Pi_0, \nabla^2 \eta\ra-
  \div_\Sigma(|  H | \nabla\eta)   +f$$
  which is in $L^2(\Sigma)$, then there exists $\psi\in W^{2,2}(\Sigma)$ such that $\Delta \psi=h$. Now we have
  \begin{equation}\label{wsol-e2}
  \int_\Sigma     \frac{\Delta \eta \Delta \phi}{ |   H | }
  =\int_\Sigma \phi\Delta \psi=\int_\Sigma \psi\Delta \phi.
  \end{equation}
  Hence
  \begin{equation}\label{wsol-e3}
  \int_\Sigma     \lf(\frac{\Delta \eta }{ |   H | }-\psi\ri)\Delta \phi=0
  \end{equation}
  for all $\phi\in \wt W^{2,2}(\Sigma)$. Recall that, for any $\zeta\in L^2(\Sigma)$ with $\int_\Sigma \zeta=0$, there is $\phi\in  \wt W^{2,2}(\Sigma)$ with $\Delta\phi=\zeta$. So \eqref{wsol-e3} implies that
  $$
  \int_\Sigma     \lf(\frac{\Delta \eta }{ |   H | }-\psi\ri)\zeta=0
  $$
  for all $\zeta\in L^2(\Sigma)$ with $\int_\Sigma \zeta=0$. Therefore,
  $$
  \frac{\Delta \eta }{ |  H | }-\psi
  = C_4
  $$
  for some constant $ C_4$.
    Since $\eta, \psi\in W^{2,2}(\Sigma)$, we know $\eta\in W^{4,2}(\Sigma)$ by \cite[Theorem 9.19]{Gilbarg-Trudinger}.
 This, together with the fact that $ \Delta \psi = h $, implies
  \be \label{lcriticaleq}
  \Delta \lf( \frac{\Delta \eta }{ |   H | } \ri) - \la H_0 \sigma - \Pi_0, \nabla^2 \eta\ra +
  \div_\Sigma(|  H | \nabla\eta)   = f .
  \ee
    If $f\in C_0^{k-4,\alpha}$, then it is easy to see that $\eta\in C_0^{k,\alpha}$ by bootstrap and the fact that $\eta\in \wt W^{4,2}(\Sigma)$. Hence, $D_\tau \mH_0|_{(0,0)}$ is surjective.
Theorem     \ref{criticalptsexist-t1}  is now proved.
\end{proof}

\begin{rem}
Suppose $ \Sigma $ is a closed connected surface in $ \R^n $ ($n\ge3$) with
second fundamental form $ \Pi_0 $ and  positive mean curvature $H_0$.
By Corollary \ref{2ndvarinrn},
the equation
\be \label{lcriticaleq-2}
  \Delta \lf( \frac{\Delta \eta }{   H_0  } \ri) - \la H_0 \sigma - \Pi_0, \nabla^2 \eta\ra +
  \div_\Sigma( H_0  \nabla\eta)   = 0
\ee
has a nontrivial kernel on $ \Sigma $ which consists of all functions $ \eta = a_0 + \sum_{i=1}^n a_i x^i $,
where $a_0, a_1, \ldots, a_n $ are arbitrary constants and $ x_1, \ldots, x_n $ are coordinate functions
on $ \R^n$.
\end{rem}

Part (2) of Theorem \ref{Mainresult} now follows directly from Theorem \ref{criticalptsexist-t1}
and Theorem \ref{Mainresult-1}.

\bibliographystyle{amsplain}

\begin{thebibliography}{10}



\bibitem {Bray} Bray, H., {\it Proof of the Riemannian Penrose inequality using the positive mass theorem}, J. Differential Geom. \textbf{59} (2001), 177-267. 

\bibitem{BY1} Brown, J. David and York, Jr., James W.,
\newblock Quasilocal energy in general relativity.
\newblock In {\em Mathematical aspects of classical field theory (Seattle, WA,   1991)}, volume 132 of {\em Contemp. Math.}, pages 129--142. Amer. Math. Soc., Providence, RI, 1992.

\bibitem{BY2}
Brown, J. David and York, Jr., James W.,
\newblock Quasilocal energy and conserved charges derived from the   gravitational action.
\newblock{\em Phys. Rev. D (3)}, \textbf{47} (4):1407--1419,1993.

\bibitem{ChenWangYau} Chen, P.-N., Wang, M.-T. and Yau, S.-T.,{\sl  Evaluating 
quasilocal energy and solving optimal embedding equation at null infinity},  
arXiv:1002.0927v2.


\bibitem{Gilbarg-Trudinger} Gilbarg, D. and  Trudinger, N. S.,
{\sl Elliptic partial differential equations of second order},
second edition,   Springer-Verlag, (1983).



\bibitem{Huang08-centermass}
Huang, L.-H., {\it On the center of mass of isolated systems with general asymptotics},
Class. Quantum. Grav. \textbf{26} (2009), 015012.

\bibitem{Huisken-Ilmanen}
Huisken, G. and Ilmanen, T, {\it The inverse mean curvature flow and the Riemannian Penrose Inequality}, J. Differential Geom. \textbf{59} (2001), 353--437.

\bibitem{Lam2010} Lam, M.-K., {\sl The graphs cases of the Riemannian positive mass and Penrose inequalities in all dimensions},  arXiv:1010.4256.


\bibitem{PeterLi93}
Li, P., {\it Lecture notes on geometric analysis},
Lecture Notes Series, 6. Seoul National University, Research Institute of Mathematics, Global Analysis Research Center, Seoul, 1993.
Also available at http://math.uci.edu/$\sim$pli/



\bibitem{LY1} Liu, C.-C.M. and Yau, S.-T., {\it Positivity of quasilocal
mass}, Phys. Rev.Lett. \textbf{90}  (2003) No. 23, 231102.

\bibitem{LY2} Liu, C.-C.M. and Yau, S.-T., {\it Positivity of quasilocal
mass II}, J. Amer.Math.Soc. \textbf{19} (2006) No. 1, 181-204.


\bibitem{Miao02}
Miao, P., {\sl Positive mass theorem on manifolds admitting corners along a hypersurface},  Adv. Theor. Math. Phys. \textbf{6} (2002), no. 6, 1163--1182 (2003).

\bibitem{Miao-LRPI}  Miao, P.,
{\sl On a Localized Riemannian Penrose Inequality },
Commun. Math. Phys. \textbf{292}  (2009), no. 1, 271-284.

\bibitem{MiaoShiTam} Miao, P.,
Shi, Y. G.  and Tam, L.-F., {\sl On geometric problems related to Brown-York and Liu-Yau quasilocal mass},
Commun. Math. Phys. \textbf{298},  (2010), 437-459.



\bibitem{Nirenberg} Nirenberg, L.,
{\sl The Weyl and Minkowski problems in differential geoemtry in the large},
Comm. Pure Appl. Math. \textbf{6} (1953), 337-394.



\bibitem{Reilly77}
Reilly, R. C., {\it Applications of the Hessian operator in a Riemannian manifold},
 Indiana Univ. Math. J.  \textbf{26} (1977), no. 3, 459--472.


\bibitem{ShiTam02}
Shi, Y.-G. and  Tam, L.-F., {\it Positive mass theorem and the boundary behaviors of compact manifolds with nonnegative scalar curvature}, J. Differential Geom. \textbf{62} (2002), 79--125.


\bibitem{WangYau-PRL} Wang, M.-T. and Yau, S.-T. \textit{Quasilocal mass in general relativity.} Phys. Rev. Lett. \textbf{102} (2009), 021101.


\bibitem{WangYau08} Wang, M. -T. and Yau, S.-T., {\sl Isometric embeddings into
the Minkowski space and new quasi-local mass}, Comm. Math. Phys. \textbf{288}(3)
(2009), 919--942.


\end{thebibliography}

\end{document}